\documentclass[10pt,a4paper,oneside]{article}

\usepackage{amsthm}
\usepackage[font=normalsize,position=top]{caption}
\newtheorem{theorem}{Theorem}[section]
\newtheorem{definition}[theorem]{Definition}

\newtheorem{lemma}[theorem]{Lemma}

\newtheorem{corollary}[theorem]{Corollary}

\usepackage{amsmath,amssymb,amsfonts}
\usepackage[usenames,dvipsnames]{color}
\usepackage[dvips]{graphicx} 
\usepackage{pstricks,pst-node,pst-tree}
\usepackage{bm}
\usepackage{multirow}
\usepackage{multicol}
\usepackage{url}
\usepackage{array}
\usepackage{arydshln}
\usepackage[caption=false]{subfig}
\usepackage{booktabs}
\usepackage{accents}
\definecolor{lightgray}{gray}{0.5}
\definecolor{darkred}{rgb}{.6,0,0}
\definecolor{darkblue}{rgb}{0,0,0.5}
\usepackage[colorlinks=true,citecolor=darkblue,urlcolor=darkblue,linkcolor=darkred]{hyperref}
\usepackage[textsize=footnotesize]{todonotes}

\newcommand{\sspcoef}{\mathcal{C}}
\newcommand{\ceff}{\sspcoef_{\textnormal{eff}}}
\newcommand{\clin}{\sspcoef^{\textnormal{lin}}_{s,q}}

\newcommand{\Dt}{\Delta t}
\newcommand{\DtFE}{\Delta t_{\textnormal{FE}}}
\newcommand{\equival}[1]{\mbox{{\large $\underaccent{\hspace{8pt}#1}{\simeq } \,$}}}

\newcommand{\nline}{\tabularnewline\noalign{\vskip 2pt}}
\newcommand{\mydashrule}{%
  \noalign{\vskip 4pt}%
  \hdashline[2pt/3pt]%
  \noalign{\vskip 4pt}}
\newcommand{\mymidrule}{\midrule\noalign{\vskip 1.5pt}}
\newcolumntype{M}[1]{>{\vspace{0.25pt}\centering\hspace{0pt}\vspace{0.25pt}}m{#1}}

\newtheorem{remark}[theorem]{Remark}
\numberwithin{theorem}{section}
\numberwithin{equation}{section}
\numberwithin{table}{section}
\numberwithin{figure}{section}

\newenvironment{smalltrees}{
	\psset{treesep=1.5ex,levelsep=1.2ex,treemode=U,showbbox=false}
	
	\renewcommand{\Ts}{\Tdot[dotstyle=*,dotscale=0.7]}
}

\psset{treesep=1.8ex,levelsep=1.8ex,treemode=U,showbbox=false}

\newcommand{\Ts}{\Tdot[dotstyle=*,dotscale=0.9]}
\newcommand{\tree}[1]{
\ifthenelse{\equal{#1}{1}}{\raisebox{0.5ex}{{\pstree{\Ts}{}}}}{}%
\ifthenelse{\equal{#1}{2}}{\raisebox{0.0ex}{{\pstree{\Ts}{\Tn\Ts}}}}{}%
\ifthenelse{\equal{#1}{3}}{\raisebox{0.0ex}{{\pstree{\Ts}{\Ts\Ts}}}}{}%
\ifthenelse{\equal{#1}{4}}{\raisebox{-0.5ex}{{\pstree{\Ts}{\Tn\pstree{\Ts}{\Ts\Tn}}}}}{}%
\ifthenelse{\equal{#1}{5}}{\raisebox{0.0ex}{{\pstree{\Ts}{\Ts\Ts\Ts}}}}{}%
\ifthenelse{\equal{#1}{6}}{\raisebox{-0.5ex}{{\pstree{\Ts}{\Ts\pstree{\Ts}{\Ts\Tn}}}}}{}%
\ifthenelse{\equal{#1}{7}}{\raisebox{-0.5ex}{{\pstree{\Ts}{\Tn\pstree{\Ts}{\Ts\Ts}}}}}{}%
\ifthenelse{\equal{#1}{8}}{\raisebox{-0.7ex}{{\pstree{\Ts}{\Tn\pstree{\Ts}{\pstree{\Ts}{\Tn\Ts}\Tn}}}}}{}%
\ifthenelse{\equal{#1}{9}}{\raisebox{0.0ex}{{\pstree{\Ts}{\Ts\Ts\Ts\Ts}}}}{}%
\ifthenelse{\equal{#1}{10}}{\raisebox{-0.3ex}{{\pstree{\Ts}{\Ts\Ts\pstree{\Ts}{\Ts\Tn}}}}}{}%
\ifthenelse{\equal{#1}{11}}{\raisebox{-0.5ex}{{\pstree{\Ts}{\Ts\pstree{\Ts}{\Ts\Ts}}}}}{}%
\ifthenelse{\equal{#1}{12}}{\raisebox{-1.0ex}{{\pstree{\Ts}{\Ts\pstree{\Ts}{\pstree{\Ts}{\Tn\Ts}\Tn}}}}}{}%
\ifthenelse{\equal{#1}{13}}{\raisebox{-0.5ex}{{\pstree{\Ts}{\pstree{\Ts}{\Ts}\pstree{\Ts}{\Ts}}}}}{}%
\ifthenelse{\equal{#1}{14}}{\raisebox{-0.5ex}{{\pstree{\Ts}{\Tn\pstree{\Ts}{\Ts\Ts\Ts}}}}}{}%
\ifthenelse{\equal{#1}{15}}{\raisebox{-0.7ex}{{\pstree{\Ts}{\Tn\pstree{\Ts}{\Ts\pstree{\Ts}{\Ts\Tn}}}}}}{}%
\ifthenelse{\equal{#1}{16}}{\raisebox{-0.7ex}{{\pstree{\Ts}{\Tn\pstree{\Ts}{\pstree{\Ts}{\Ts\Ts}\Tn}}}}}{}%
\ifthenelse{\equal{#1}{17}}{\raisebox{-1.3ex}{{\pstree{\Ts}{\Tn\pstree{\Ts}{\pstree{\Ts}{\Tn\pstree{\Ts}{\Ts\Tn}}\Tn}}}}}{}%
}

\title{Strong stability preserving explicit Runge--Kutta methods of maximal effective order
}
\author{
        Yiannis Hadjimichael\thanks{Computer, Electrical and Mathematical Sciences and 
        Engineering Division, King Abdullah University of Science \& Technology (KAUST), 
        P.O. Box 4700, Thuwal 23955, Saudi Arabia
        (\url{yiannis.hadjimichael@kaust.edu.sa}, 
        \url{david.ketcheson@kaust.edu.sa}).
        The work of these authors is supported by Award No. FIC/2010/05, made by King 
        Abdullah University of Science and Technology (KAUST).}
        \and 
        Colin B.~Macdonald\thanks{Mathematical Institute, University of Oxford, OX1\,3LB, UK 
        (\url{macdonald@maths.ox.ac.uk}).
        The work of this author was supported by NSERC 
        Canada and by Award No KUK-C1-013-04 made by King Abdullah University of Science 
        and Technology (KAUST).}
        \and 
        David I.~Ketcheson\footnotemark[1]
        \and 
        James H.~Verner\thanks{Department of Mathematics, Simon Fraser University,
        Burnaby, British Columbia, V5A\,1S6, Canada
        (\url{jverner@pims.math.ca}).
        The work of this author was supported by Simon Fraser University.}
}
\date{}

\begin{document}

        \maketitle
        
        \begin{abstract}
                We apply the concept of effective order to strong stability preserving 
                (SSP) explicit Runge--Kutta methods.
                Relative to classical Runge--Kutta methods, methods with an effective order of accuracy
                are designed to satisfy a relaxed set of order conditions, but yield higher 
                order accuracy  when composed with special starting and stopping methods. 
         We show that this allows the construction of four-stage SSP methods with 
         effective order four (such methods cannot have classical order four). 
         However, we also prove that effective order five methods---like classical
         order five methods---require the use of non-positive weights and so cannot
         be SSP.
         By numerical optimization, we construct explicit SSP Runge--Kutta methods 
         up to effective order four and establish the optimality of many of them.
                Numerical experiments demonstrate the validity of these methods in
                practice.
        \end{abstract}
		
%
%
		
        \section{Introduction}\label{sec:Intro}
Strong stability preserving time discretization methods were originally
developed for the solution of nonlinear hyperbolic 
partial differential equations (PDEs).  Solutions of such PDEs may 
contain discontinuities even when the initial conditions are smooth.
Many numerical methods for their solution are based on a method-of-lines approach 
in which the problem is first discretized in space to yield a system of ODEs. 
The spatial discretization is often chosen to ensure the solution is total variation diminishing (TVD),
in order to avoid the appearance of spurious oscillations near discontinuities,
\emph{when coupled with first-order forward Euler time integration}.
Strong stability preserving (SSP) time discretizations (also known as TVD
discretizations \cite{Gottlieb/Shu:1998}) are high-order time
discretizations that guarantee the TVD property (or other convex functional bounds), with a
possibly different step-size restriction \cite{Gottlieb2011a}.
Section~\ref{sec:SSP} reviews Runge--Kutta methods and the concept of
strong stability preserving methods.

Explicit SSP Runge--Kutta methods cannot have order greater
than four \cite{Ruuth2002}.  However, a Runge--Kutta method may 
achieve an effective order of accuracy higher than its classical order 
by the use of special starting and stopping procedures.
The conditions for a method to have effective order $q$ are 
in general less restrictive than the conditions for a method
to have classical order $q$.
Section~\ref{sec:Algebraic_RK} presents a brief overview of the
algebraic representation of Runge--Kutta methods, following Butcher
\cite{Butcher2008_book}.
This includes the concept of effective order and a list of effective
order conditions.

We examine the SSP properties of explicit Runge--Kutta methods whose
effective order is greater than their classical order.
Previous studies of SSP Runge--Kutta methods have considered only the
classical order of the methods.
Three natural questions are:
\begin{itemize}
    \item Can an SSP Runge--Kutta method have effective order of accuracy greater than four?
    \item If we only require methods to have {\em effective} order $q$, is it possible to achieve larger
            SSP coefficients than those obtained in methods with {\em classical} order $q$?
    \item SSP Runge--Kutta methods of order four require at least five stages.  Can SSP methods of 
          effective order four have fewer stages?
          
\end{itemize}
We show in Section \ref{sec:ExRK_barrier} that the answer to the first question is negative.
We answer the second question by numerically solving the problem of optimizing
the SSP coefficient over the class of methods with effective order $q$;
see Section~\ref{sec:optimal_ESSPRK}.  
Most of the methods we find are shown to be optimal, as they achieve a certain theoretical
upper bound on the SSP coefficient that is obtained by considering only
linear problems \cite{Kraaijevanger1986}.
We answer the last question affirmatively
by construction, also in Section~\ref{sec:optimal_ESSPRK}.
The paper concludes with numerical experiments in
Section~\ref{sec:numerics} and conclusions in
Section~\ref{sec:Conclusion}.

        \section{Strong stability preserving Runge--Kutta methods}\label{sec:SSP}
Strong stability preserving (SSP) time-stepping methods were originally introduced
for time integration of systems of hyperbolic conservation laws
\cite{Shu/Osher:1988} 
\begin{align}\label{eq:pde}
	\bm{U}_t + \nabla \cdot \bm{f}(\bm{U}) = 0,   
\end{align}
with appropriate initial and boundary conditions.
A spatial discretization gives the system of ODEs
\begin{align}\label{eq:ode_system}
    \bm{u}'(t) = \bm{F}(\bm{u}(t)),
\end{align}
where $\bm{u}$ is a vector of continuous-in-time grid values approximating 
the solution $\bm{U}$ at discrete grid points.
Of course, \eqref{eq:ode_system} can arise in many ways and $\bm{F}$
need not necessarily represent a spatial discretization.
Particularly, $\bm{F}$ may be time-dependent, but we can always make a 
transformation to an autonomous form.
In any case, a time discretization then produces a sequence of
solutions $\bm{u}^{n} \approx \bm{u}(t_n)$.
This work studies explicit Runge--Kutta time discretizations.
An explicit $s$-stage Runge--Kutta method takes the form
\begin{align*}
	\bm{u}^{n+1} &= \bm{u}^{n} + \Dt \sum_i^s b_i \bm{F}(\bm{Y}_i), 
\end{align*}
where
\begin{align*}
	\bm{Y}_i &= \bm{u}^{n} + \Dt \sum_j^{i-1} a_{ij} \bm{F}(\bm{Y}_j).
\end{align*}
Such methods are characterized by the coefficient matrix $A = (a_{ij}) \in 
\mathbb{R}^{s \times s}$, the weight vector $\bm{b} = (b_i) \in \mathbb{R}^s$
and the abscissa $\bm{c} = (c_i) \in \mathbb{R}^s$, where 
$c_i = \sum_{j=1}^{i-1}a_{ij}$.
The accuracy and stability of the method depend on the coefficients of the 
Butcher tableau $(A,\bm{b},\bm{c})$ \cite{Butcher2008_book}.

In some cases, the solutions of hyperbolic conservation laws satisfy a 
monotonicity property. For example, if \eqref{eq:pde} is scalar then solutions 
are monotonic in the total variation semi-norm \cite{Ketcheson2008}.
For this reason, many popular spatial discretizations are designed such 
that, for a suitable class of problems, the solution $\bm{u}$ in 
\eqref{eq:ode_system} computed with the forward Euler scheme is
non-increasing (in time) in some norm, semi-norm, or convex functional; i.e.,
\begin{align}\label{eq:forwardEuler}
    \|\bm{u} + \Dt\bm{F}(\bm{u})\| \le \|\bm{u}\|, \quad \text{for all } \bm{u} \text{ and for } 0 \le \Dt \le \DtFE.
\end{align}
If this is the case, then an SSP method also generates a solution whose norm is
non-increasing in time, under a modified time-step restriction.
\begin{definition}[Strong Stability Preserving]
	A Runge--Kutta method is said to be \emph{strong stability preserving} with
	\emph{SSP coefficient} $\sspcoef > 0$ if, whenever the forward Euler condition
	\eqref{eq:forwardEuler} holds and 
	\begin{align*}
		0 \leq \Dt \leq \sspcoef \DtFE,
	\end{align*}
	the Runge--Kutta method generates a monotonic sequence of solution values $\bm{u}^n$ satisfying
	\begin{align*}
  		\|\bm{u}^{n+1}\| \le \|\bm{u}^n\|.
	\end{align*}
\end{definition}

Note that $\DtFE$ is a property of the spatial discretization $\bm{F}$
and is independent of $\bm{u}$.
The SSP coefficient $\sspcoef$ is a property of the particular
time-stepping method and quantifies the allowable time step size relative 
to that of the forward Euler method.
Generally we want the SSP coefficient to be as large as possible for efficiency.
To allow a fair comparison of explicit methods with different number of stages, 
we consider the \emph{effective SSP coefficient}
\begin{align*}
	\ceff = \frac{\sspcoef}{s}.
\end{align*}
Note that the use of the word \emph{effective} here is unrelated to the 
concept of \emph{effective order} introduced in Section~\ref{sec:Algebraic_RK}.

\subsection{Optimal SSP schemes}\label{subsec:Optimal_SSPRK}
We say that an SSP Runge--Kutta method is optimal if it has the largest 
possible SSP coefficient for a given order and a given number of stages.
The search for these optimal methods was originally based on
expressing the Runge--Kutta method as combinations of forward Euler
steps (the Shu--Osher form) and solving a non-linear optimization
problem \cite{Gottlieb/Shu:1998, Gottlieb2001, Spiteri2003a, Spiteri2003b, 
Ruuth2004, Ruuth:2006}.
However, the SSP coefficient is related to the 
\emph{radius of absolute monotonicity} \cite{Kraaijevanger1991} and, 
for irreducible Runge--Kutta methods, the two are equivalent 
\cite{Ferracina2004, Higueras2004}.
This gives a simplified algebraic characterization of the SSP coefficient
\cite{Ferracina2005}; it is the maximum value of $r$ such that the following
conditions hold:
\begin{subequations} \label{eq:absmon}
\begin{align}
    K(I + rA)^{-1} \geq 0 \\
    \bm{e}_{s+1} - rK(I + rA)^{-1}\bm{e}_{s} \geq 0,
\end{align}
\end{subequations}
provided that $I + rA$ is invertible.
Here
\begin{equation*}
    K = \left(
            \begin{array}{c}
                     A              \\
                     \bm{b}^{\texttt{T}}
            \end{array}
         \right),
\end{equation*}
while $\bm{e}_s$ denotes the vector of ones of length $s$ and $I$ is the
$s \times s$ identity matrix.
The inequalities are understood component-wise.

The optimization problem of finding optimal SSP Runge--Kutta methods
can thus be written as follows:
\begin{equation}\label{eq:SSP_opt}
    \max_{A, \bm{b}, r} \; r \quad \text{subject to \eqref{eq:absmon} and } \Phi(K) = 0.
\end{equation}
Here \( \Phi(K) \) represents the  order conditions.

Following \cite{Ketcheson2008, Ketcheson/Macdonald/Gottlieb:2009}, 
we will numerically solve the optimization problem \eqref{eq:SSP_opt} to find
optimal explicit SSP Runge--Kutta methods for various effective orders of accuracy.
However, we first need to define the order conditions $\Phi(K)$ for these methods.
This is discussed in the next section.

        \section{The effective order of Runge--Kutta methods}\label{sec:Algebraic_RK}

The definition, construction, and application of methods with an
effective order of accuracy relies on the use of starting and stopping
methods.
Specifically, we consider a \emph{starting method} $S$, a \emph{main
  method} $M$, and a \emph{stopping method} $S^{-1}$.
The successive use of these three methods results in a method $P =
S^{-1}MS$, which denotes the application of method $S$, followed by
method $M$, followed by method $S^{-1}$.
The method $S^{-1}$ is an ``inverse'' of method $S$.
We want $P$ to have order $q$, whereas $M$ might have lower classical
order $p < q$.
We then say $M$ has \emph{effective order} $q$.

When the method $P$ is used for $n$ steps,
$$P^n = (S^{-1}MS)^n = (S^{-1}MS) \cdots (S^{-1}MS) (S^{-1}MS),$$
it turns out that only $M$ need be used repeatedly, as in
$S^{-1} M^n S$,
because 
$S S^{-1}$ leaves the solution unchanged up to order $q$.
The starting method introduces a perturbation to the solution,
followed by $n$ time steps of the main method $M$, and finally the
stopping method is used to correct the solution.
In Section~\ref{subsec:starting_stopping}, we propose alternative
starting and stopping procedures which allow the overall procedure to
be SSP.

The effective order of a Runge--Kutta method is defined in an abstract 
algebraic context introduced by Butcher \cite{Butcher1969} and developed 
further in \cite{Butcher1972, Hairer1974, Butcher1996, Butcher1998} and 
others.
We follow the book \cite{Butcher2008_book} in our 
derivation of the effective order conditions.

\subsection{The algebraic representation of Runge--Kutta methods}\label{subsec:Algebraic_representation}

\begin{table}
	\caption{Elementary weights $\alpha(t_i)$ of trees $t_i$ up to order five for a
  		Runge--Kutta method with Butcher tableau $(A,\bm{b},\bm{c})$. 
  		Here $C$ is a diagonal matrix with components 
  		$c_{i} = \sum_{j=1}^{i-1} a_{ij}$ and exponents of vectors 
  		represent component exponentiation.
  		By convention $\alpha_0 = \alpha(t_{0}) = 1$, where $t_{0}$ 
  		denotes the empty tree.}
	\centering
	\begin{smalltrees}
		\begin{tabular}{ccccccccc}
    		\cline{1-4}\cline{6-9}
                \noalign{\vskip 2pt}
    		$i$ & tree $t_i$ & $\alpha(t_i)$ & $\gamma(t_i)$ & & $i$ & tree $t_i$ & $\alpha(t_i)$ & $\gamma(t_i)$ \\
                \noalign{\vskip 1pt}
    		\cline{1-4}\cline{6-9}
                \noalign{\vskip 3pt}
    		0 & $\emptyset$ \hspace{15pt}  & 1 & 0 & & 9 & \hspace{15pt} \tree{9} & $\bm{b}^T\bm{c}^4$ & 5 \\
    		1 & \hspace{15pt} \tree{1} & $\bm{b}^T\bm{e}$ & 1 & & 10 & \tree{10} \hspace{15pt} & $\bm{b}^TC^2A\bm{c}$ & 10 \\
    		2 & \tree{2} \hspace{15pt}  & $\bm{b}^T\bm{c}$ & 2 & & 11 & \hspace{15pt} \tree{11} & $\bm{b}^TCA\bm{c}^2$ & 15 \\
    		3 & \hspace{15pt} \tree{3} & $\bm{b}^T\bm{c}^2$ & 3 & & 12 & \tree{12} \hspace{15pt} & $\bm{b}^TCA^2\bm{c}$ & 30 \\
    		4 & \tree{4} \hspace{15pt}  & $\bm{b}^TA\bm{c}$ & 6 & & 13 & \hspace{15pt} \tree{13} & $\bm{b}^T(A\bm{c})^2$ & 20 \\
    		5 & \hspace{15pt} \tree{5} & $\bm{b}^T\bm{c}^3$ & 4 & & 14 & \tree{14} \hspace{15pt} & $\bm{b}^TA\bm{c}^3$ & 20 \\
    		6 & \tree{6} \hspace{15pt}  & $\bm{b}^TCA\bm{c}$ & 8 & & 15 & \hspace{15pt} \tree{15} & $\bm{b}^TACA\bm{c}$ & 40 \\
    		7 & \hspace{15pt} \tree{7} & $\bm{b}^TA\bm{c}^2$ & 12 & & 16 & \tree{16} \hspace{15pt} & $\bm{b}^TA^2\bm{c}^2$ & 60 \\
    		8 & \tree{8} \hspace{15pt}  & $\bm{b}^TA^2\bm{c}$ & 24 & & 17 & \hspace{15pt} \tree{17} & $\bm{b}^TA^3\bm{c}$ & 120 \nline
		\noalign{\vskip 3pt}    		
    		\cline{1-4}\cline{6-9}
  		\end{tabular}
  \end{smalltrees}
	\label{tab:elementary_weights}
\end{table}

According to Butcher's algebraic theory, irreducible Runge--Kutta methods
are placed in one-to-one correspondence with elements of a group
$G$, consisting of real-valued functions on the set of rooted trees \cite[Theorem~384A]{Butcher2008_book}.
A Runge--Kutta method corresponds to the map that takes each rooted tree $t$
to the corresponding elementary weight $\Phi(t)$ of that Runge--Kutta method.
Table~\ref{tab:elementary_weights} lists the elementary weights for trees of
up to degree five; a general recursive formula can be found in
\cite[Definition~312A]{Butcher2008_book}.
The ordering of trees given in Table~\ref{tab:elementary_weights} is used
throughout the remainder of this work; thus $t_9$ refers to the tree with
elementary weight $\bm{b}^T \bm{c}^4$.
For a function $\alpha \in G$ we write the values of the
elementary weights as $\alpha_{i} = \alpha(t_{i})$ for tree $t_{i}$.
A special element of the group $E \in G$ corresponds to the
(hypothetical) method that evolves the solution exactly.
The values of $E(t)$ are denoted $1/\gamma(t)$ \cite{Butcher2008_book}
and the values of $\gamma(t)$ are included in
Table~\ref{tab:elementary_weights}.
Classical order conditions are obtained by comparing the elementary weights 
of a method with these values.

Let $\alpha, \beta \in G$ correspond to Runge--Kutta methods $M_1$ and $M_2$
respectively.
The application of method $M_1$ followed by method $M_2$ corresponds to
the multiplicative group operation $\alpha\beta$.\footnote{We write
	$M_2M_1$ to mean the application of $M_1$
	followed by the application of $M_2$
	(following matrix and operator ordering convention)
	but when referring to products of elements of $G$
        we use the reverse ordering ($\alpha\beta$)
	to match the convention in \cite{Butcher2008_book}.}
This product is defined by partitioning the input tree and computing
over the resulting forest \cite[\S~383]{Butcher2008_book}.

Two Runge--Kutta methods $M_1$ and $M_2$, are equivalent up to order
$p$ if their corresponding elements in $G$, $\alpha$ and $\beta$, satisfy
$\alpha(t) = \beta(t)$, for every tree $t$ with $r(t) \leq p$,
where $r(t)$ denotes the order of the tree (number of vertices).
We denote this equivalence relation by
$$M_1 \equival{p} M_2.$$
In this sense, methods have inverses: the product of $\alpha^{-1}$ and
$\alpha$ must match the identity method up to order $p$.
Note that inverse methods up to order $p$ are not unique and inverse 
methods of explicit methods need not be implicit.
We can then define the effective order of accuracy of a method $M$
with starting method $S$ and stopping method $S^{-1}$. 
\begin{definition}\cite[\S~389]{Butcher2008_book}\label{def:Effective_order}
  Suppose $M$ is a Runge--Kutta method with corresponding $\alpha \in G$.
  Then the method $M$ is of effective order $q$ if there exist methods
  $S,S^{-1}$ (with corresponding $\beta, \beta^{-1} \in G$) such that
	\begin{equation}\label{eq:Effective_order_1}
		(\beta\alpha\beta^{-1})(t) = E(t), \; \text{for every tree with } r(t) \leq q,
	\end{equation}
        where $\beta^{-1}$ is an inverse of $\beta$ up to order $q$; i.e.
        \begin{align*}
        		(\beta^{-1}\beta)(t) = 1(t), \; \text{for every tree with } r(t) \leq q.
		\end{align*}	        
        Here $1\in G$ is the identity element
        and $E\in G$ is the exact evolution operator.
\end{definition}

\begin{table}
    \caption{Effective order five conditions on $\alpha$ (main
      method $M$) in terms of order conditions on $\beta$
      (starting method $S$).
      See also \cite[\S~389]{Butcher2008_book}.
      Recall that $\alpha_i$ and $\beta_i$
      are the elementary weights associated with the index $i$ in
      Table~\ref{tab:elementary_weights}.
      We assume that $\beta_1=0$ (see
      Section~\ref{subsubsec:Main_starting_conditions}).}
  \small
  \setlength{\extrarowheight}{0.5pt}
  \centering
  \begin{tabular}{p{2mm}p{51mm}p{64mm}}
    \toprule
    $q$  &  Effective order conditions \\
    \mymidrule
    $1$  &
            $\alpha_1  = 1$. \\
    \mydashrule
    $2$  &
            $\alpha_2  = \tfrac{1}{2}$. \\
   \mydashrule
    $3$  &
            $\alpha_3  = \tfrac{1}{3} + 2\beta_2$,  \quad
            $\alpha_4  = \tfrac{1}{6}$.  \\
    \mydashrule
    $4$  &  \multicolumn{2}{l}{%
            $\alpha_5  = \tfrac{1}{4} + 3\beta_2 + 3\beta_3$, \hfill
            $\alpha_6  = \tfrac{1}{8} + \beta_2 + \beta_3 + \beta_4$, \hfill
            $\alpha_7  = \tfrac{1}{12} +\beta_2 - \beta_3 + 2\beta_4$, \hfill
            $\alpha_8  = \tfrac{1}{24}$.} \\
    \mydashrule
    $5$  &
            $\alpha_9  = \tfrac{1}{5} + 4\beta_2 + 6\beta_3 + 4\beta_5$,
         &  \hspace*{-3pt}$\alpha_{10} = \tfrac{1}{10} + \tfrac{5}{3}\beta_2 - 2\beta_2^{2} + \tfrac{5}{2}\beta_3 + \beta_4 + \beta_5 + 2\beta_6$, \nline
         &  $\alpha_{11} = \tfrac{1}{15} + \tfrac{4}{3}\beta_2 + \tfrac{1}{2}\beta_3 + 2\beta_4 + 2\beta_6 + \beta_7$, 
         &  \hspace*{-3pt}$\alpha_{12} = \tfrac{1}{30} + \tfrac{1}{3}\beta_2 - 2\beta_2^{2} + \tfrac{1}{2}\beta_3 + \tfrac{1}{2}\beta_4 + \beta_6 + \beta_8$, \nline
         &  $\alpha_{13} = \tfrac{1}{20} + \tfrac{2}{3}\beta_2 - \beta_2^{2} + \beta_3 + \beta_4 + 2\beta_6$,
         &  \hspace*{-3pt}$\alpha_{14} = \tfrac{1}{20} + \beta_2 + 3\beta_4 - \beta_5 + 3\beta_7$, \nline
         &  $\alpha_{15} = \tfrac{1}{40} + \tfrac{1}{3}\beta_2 + \tfrac{3}{2}\beta_4 - \beta_6 + \beta_7 + \beta_8$,
         &  \hspace*{-3pt}$\alpha_{16} = \tfrac{1}{60} + \tfrac{1}{3}\beta_2 - \tfrac{1}{2}\beta_3 + \beta_4 - \beta_7 + 2\beta_8$, \,\! $\alpha_{17} = \tfrac{1}{120}$. \nline
            \bottomrule
    \end{tabular}
    \label{tab:effective_OCs_on_alpha}
\end{table}

\subsection{Effective order conditions}\label{sec:effOrderCond}
For the main method $M$ to have effective order $q$, its coefficients
and those of the starting and stopping methods must satisfy a set of
algebraic conditions.
These \emph{effective order conditions} can be found
by rewriting \eqref{eq:Effective_order_1} as
$(\beta\alpha)(t) = (E\beta)(t)$  
and applying the group product operation.
For trees up to order five these are tabulated in Table~\ref{tab:effective_OCs_on_alpha} (and also in \cite[\S~389]{Butcher2008_book}).
In general, the effective order conditions allow more degrees of
freedom for method design than do the classical order conditions.
Note that the effective order conditions match the classical order conditions up to
second order.
\begin{remark}\label{rem:talltrees}
	The effective order conditions of the main method for the ``tall" trees 
	$t_1, t_2, t_4, t_8, t_{17}, \dots$ match the classical order conditions
	and these are precisely the order conditions for linear problems.
	This follows from inductive application of the group product
	on the tall trees. 
	Therefore, methods of effective order $q$ have classical order at least $q$ 
	for linear problems.
\end{remark}


\subsubsection{Order conditions of the main and starting methods}\label{subsubsec:Main_starting_conditions}
As recommended in \cite{Butcher2008_book},
we consider the elementary weights $\beta_{i}$ of the starting method as free
parameters when determining the elementary weights $\alpha_i$ of the main method.
The relationship in Table~\ref{tab:effective_OCs_on_alpha} between the
$\alpha_i$ and $\beta_i$ is mostly linear (although there are a few
$\beta_2^2$ terms).
It is thus straightforward to (mostly) isolate the equations for $\alpha_i$
and determine the $\beta_i$ as linear combination of the $\alpha_i$.
This separation provides maximal degrees of freedom and minimizes the number of
constraints when constructing the method $M$.
The resulting effective order conditions for the main method $M$ are given
in Table~\ref{tab:effective_OCs} (up to effective order five).
For a specified classical and effective order, these are the equality constraints
$\Phi(K)$ in the optimization problem \eqref{eq:SSP_opt} for method $M$.

Constructing the main method $M$ then determines the $\alpha$ values
and we obtain a set of order conditions on $\beta$ (for that
particular choice of $M$).  These are given in the right-half of
Table~\ref{tab:effective_OCs}.
We can also find the order conditions on $S^{-1}$ in terms of the
$\beta_i$ (see \cite[Table~386(III)]{Butcher2008_book}).
We note that increasing the classical order of the main method requires 
$\alpha_i = 1/\gamma(t_i)$ and thus by Table~\ref{tab:effective_OCs_on_alpha} requires more of the $\beta_i$ to be zero.

\begin{table}
    \caption{Effective order $q$, classical order $p$ conditions on $ \alpha $ and $ \beta $ for the main and starting methods, $M$ and $S$ respectively.}
  \small
  \centering
    \begin{tabular}{M{2mm}M{2mm}M{68mm}M{67mm}}
      \toprule
        $q$ & $p$ & Order conditions for main method $M$ & Order conditions for starting method $S$ \nline
      \midrule
        \multirow{1}{*}{3} & \multirow{1}{*}{2} & {\small $\alpha_1 = 1$, $\alpha_2 = \frac{1}{2}$, $\alpha_4 = \frac{1}{6}$.} & {\small $\beta_1 = 0$, $\beta_2 = - \frac{1}{6} + \frac{1}{2}\alpha_3$.}\nline
      \mydashrule
        \multirow{3}{*}{4} & \multirow{3}{*}{2} & {\small $\alpha_1 = 1$, $\alpha_2 = \frac{1}{2}$, $\alpha_4 = \frac{1}{6}$,} & {\small $\beta_1 = 0$, $\beta_2 = - \frac{1}{6} + \frac{1}{2}\alpha_3$,}\nline
        & & {\small $\frac{1}{4} - \alpha_3 + \alpha_5 - 2\alpha_6 + \alpha_7 = 0$, $\alpha_8 = \frac{1}{24}$.} & {\small $\beta_3 = \frac{1}{12} - \frac{1}{2}\alpha_3 + \frac{1}{3}\alpha_5$, $\beta_4 = - \frac{1}{24} - \frac{1}{3}\alpha_5 + \alpha_6$.} \nline
      \mydashrule
        \multirow{3}{*}{4} & \multirow{3}{*}{3} & {\small $\alpha_1 = 1$, $\alpha_2 = \frac{1}{2}$, $\alpha_3 = \frac{1}{3}$, $\alpha_4 = \frac{1}{6}$,} & {\small $\beta_1 = 0$, $\beta_2 = 0$, $\beta_3 = - \frac{1}{12}  + \frac{1}{3}\alpha_5$,} \nline
        & & {\small $\frac{1}{12} - \alpha_5 + 2\alpha_6 - \alpha_7 = 0$, $\alpha_8 = \frac{1}{24}$.} & {\small $\beta_4 = - \frac{1}{24} - \frac{1}{3}\alpha_5 + \alpha_6$.} \nline
      \mydashrule
        \multirow{8}{*}{5} & \multirow{8}{*}{2} & {\small $\alpha_1 = 1$, $\alpha_2 = \frac{1}{2}$, $\alpha_4 = \frac{1}{6}$, $\alpha_8 = \frac{1}{24}$, $\alpha_{17} = \frac{1}{120}$,} & {\small $\beta_1 = 0$, $\beta_2 = - \frac{1}{6} + \frac{1}{2}\alpha_3$,} \nline
        & & {\small $\frac{1}{4} - \alpha_3 + \alpha_5 - 2\alpha_6 + \alpha_7 = 0$,} & {\small $\beta_3 = \frac{1}{12} - \frac{1}{2}\alpha_3 + \frac{1}{3}\alpha_5$, $\beta_4 = -\frac{1}{24} - \frac{1}{3}\alpha_5 + \alpha_6$} \nline
        & & {\small $\frac{1}{4}\alpha_9-\alpha_{10}+\alpha_{13}=\beta_2^{2}$, \: $\beta_2 = - \frac{1}{6} + \frac{1}{2}\alpha_3$,} & {\small $\beta_5 = -\frac{1}{120} + \frac{1}{4}\alpha_3 - \frac{1}{2}\alpha_5 + \frac{1}{4}\alpha_9$,} \nline
        & & {\small $\frac{3}{10} - \frac{3}{2}\alpha_3 + \alpha_5 + \frac{1}{2}\alpha_9 - 3\alpha_{10} + 3\alpha_{11} - \alpha_{14} = 6\beta_2^{2}$,} & {\small $\beta_6 = \frac{7}{720} + \beta_2^{2} + \frac{1}{12}\alpha_3 - \frac{1}{2}\alpha_6 - \frac{1}{8}\alpha_9 + \frac{1}{2}\alpha_{10}$,} \nline
        & & {\small $\frac{1}{15} - \frac{1}{2}\alpha_3 + \alpha_6 + \frac{1}{2}\alpha_9 - 2\alpha_{10} + \alpha_{11} + \alpha_{12} - \alpha_{15} = 2\beta_2^{2}$,} & {\small $\beta_7 = \frac{8}{45} - 2\beta_2^{2} - \frac{7}{12}\alpha_3 + \frac{1}{2}\alpha_5 - \alpha_6 + \frac{1}{4}\alpha_9 - \alpha_{10} + \alpha_{11}$,} \nline
        & & {\small $\frac{19}{60} - \alpha_3 + \alpha_5 - 2\alpha_6 + \alpha_{11} - 2\alpha_{12} + \alpha_{16} = 4\beta_2^{2}$.} & {\small $\beta_8 = -\frac{1}{120} + \beta_2^{2} + \frac{1}{8}\alpha_9 - \frac{1}{2}\alpha_{10} + \alpha_{12}$.} \nline
      \mydashrule
        \multirow{7}{*}{5} & \multirow{7}{*}{3} & {\small $\alpha_1 = 1$, $\alpha_2 = \frac{1}{2}$, $\alpha_3 = \frac{1}{3}$, $\alpha_4 = \frac{1}{6}$, $\alpha_8 = \frac{1}{24}$,} & {\small $\beta_1 = 0$, $\beta_2 = 0$, $\beta_3 = -\frac{1}{12} + \frac{1}{3}\alpha_5$} \nline
        & & {\small $\alpha_{17} = \frac{1}{120}$, $\frac{1}{12} - \alpha_5 + 2\alpha_6 - \alpha_7 = 0$,} & {\small $\beta_4 = -\frac{1}{24} - \frac{1}{3}\alpha_5 + \alpha_6$,} \nline
        & & {\small $\frac{1}{4}\alpha_9 - \alpha_{10} + \alpha_{13} = 0$,} & {\small $\beta_5 = \frac{3}{40} - \frac{1}{2}\alpha_5 + \frac{1}{4}\alpha_9$,} \nline
        & & {\small $\frac{1}{5} - \alpha_5 - \frac{1}{2}\alpha_9 + 3\alpha_{10} - 3\alpha_{11} + \alpha_{14} = 0$,} & {\small $\beta_6 = \frac{3}{80} - \frac{1}{2}\alpha_6 - \frac{1}{8}\alpha_9 + \frac{1}{2}\alpha_{10}$,} \nline
        & & {\small $\frac{1}{10} - \alpha_6 - \frac{1}{2}\alpha_9 + 2\alpha_{10} - \alpha_{11} - \alpha_{12} + \alpha_{15} = 0$,} & {\small $\beta_7 = -\frac{1}{60} + \frac{1}{2}\alpha_5 - \alpha_6 + \frac{1}{4}\alpha_9 - \alpha_{10} + \alpha_{11}$,} \nline
        & & {\small $\frac{1}{60} - \alpha_5 + 2\alpha_6 - \alpha_{11} + 2\alpha_{12} - \alpha_{16} = 0$.} & {\small $\beta_8 = -\frac{1}{120} + \frac{1}{8}\alpha_9 - \frac{1}{2}\alpha_{10} + \alpha_{12}$.} \nline
      \mydashrule
        \multirow{7}{*}{5} & \multirow{7}{*}{4} & {\small $\alpha_1 = 1$, $\alpha_2 = \frac{1}{2}$, $\alpha_3 = \frac{1}{3}$, $\alpha_4 = \frac{1}{6}$, $\alpha_5 = \frac{1}{4}$,} & {\small $\beta_1 = 0$, $\beta_2 = 0$,} \nline
        & & {\small $\alpha_6 = \frac{1}{8}$, $\alpha_7 = \frac{1}{12}$, $\alpha_8 = \frac{1}{24}$, $\alpha_{17} = \frac{1}{120}$,} & {\small $\beta_3 = 0$, $\beta_4 = 0$,} \nline
        & & {\small $\frac{1}{4}\alpha_9 - \alpha_{10} + \alpha_{13} = 0$,} & {\small $\beta_5 = -\frac{1}{20} + \frac{1}{4}\alpha_9$,} \nline
        & & {\small $\frac{1}{20} + \frac{1}{2}\alpha_9 - 3\alpha_{10} + 3\alpha_{11} - \alpha_{14} = 0$,} & {\small $\beta_6 = -\frac{1}{40} - \frac{1}{8}\alpha_9 + \frac{1}{2}\alpha_{10}$,} \nline
        & & {\small $\frac{1}{40} + \frac{1}{2}\alpha_9 - 2\alpha_{10} + \alpha_{11} + \alpha_{12} - \alpha_{15} = 0$,} & {\small $\beta_7 = -\frac{1}{60} + \frac{1}{4}\alpha_9 - \alpha_{10} + \alpha_{11}$,} \nline
        & & {\small $\frac{1}{60} - \alpha_{11} + 2\alpha_{12} - \alpha_{16} = 0$.} & {\small $\beta_8 = -\frac{1}{120} + \frac{1}{8}\alpha_9 - \frac{1}{2}\alpha_{10} + \alpha_{12}$.} \nline
        \bottomrule
    \end{tabular}
    \label{tab:effective_OCs}
\end{table}

Tables~\ref{tab:effective_OCs_on_alpha} and~\ref{tab:effective_OCs}
both assume that $\beta_1=0$ (i.e., the starting and stopping methods
perturb the solution but do not advance the solution in time).  This
assumption is without loss of generality following \cite[Lemma
389A]{Butcher2008_book}, the proof of which shows that we can always find 
starting procedures with $\beta_1 = 0$ for which the main method has 
effective order $q$, whenever this holds for a starting method with 
$\beta_1 \neq 0$.

        \section{Explicit SSP Runge--Kutta methods have effective order at most four}\label{sec:ExRK_barrier}
The classical order of any explicit SSP Runge--Kutta method cannot be greater
than four \cite{Ruuth2002}.
It turns out that the effective order of any explicit SSP Runge--Kutta method 
also cannot be greater than four, although the proof of this result is more
involved.
We begin by recalling a well-known result.
\begin{lemma}\label{lem:positive_b}(see \cite[Theorem~4.2]{Kraaijevanger1991},\cite[Lemma 4.2]{Ruuth2002})
	Any irreducible Runge--Kutta method with positive SSP coefficient $\sspcoef>0$
	must have positive weights $\bm{b} > \bm{0}$.
\end{lemma}

Irreducibility \cite{dahlquist2006} is technically important in this
result and those that follow because a reducible SSP method might not
have positive weights (but it would be reducible to one that does, as
per the lemma).
The main result of this section is:
\begin{theorem}\label{thm:effective_barrier}
	Any explicit Runge--Kutta method with positive weights $\bm{b} > \bm{0}$ 
	has effective order at most four.
\end{theorem}


\begin{proof}
	Any method of effective order five must have classical order at least two
	(see \cite{Butcher2008_book} or Table~\ref{tab:effective_OCs}).
    Thus it is sufficient to show that any method with all positive weights
    cannot satisfy the conditions of effective order five and classical order two.

    Let $(A,\bm{b},\bm{c})$ denote the coefficients of an explicit Runge--Kutta method with
    effective order at least five, classical order at least two, and positive 
    weights $\bm{b} > \bm{0}$.
    The effective order five and classical order two conditions
    (see Table~\ref{tab:effective_OCs} with $q=5$ and $p=2$) include the following:
    \begin{subequations}\label{eq:theorem_cond}
    		\begin{align}
    			\bm{b}^T\bm{e} & = 1, \label{eq:theorem_cond_a} \\
             	\bm{b}^TA\bm{c} &= \frac{1}{6}, \label{eq:theorem_cond_b} \\
            	\frac{1}{2}\bm{b}^T\bm{c}^2 - \frac{1}{6} &= \beta_2, \label{eq:theorem_cond_c} \\
            	\frac{1}{4}\bm{b}^T\bm{c}^4 - \bm{b}^TC^2A\bm{c} + \bm{b}^T(A\bm{c})^2 &= \beta_2^2, \label{eq:theorem_cond_d}
        	\end{align}
	\end{subequations}
	where the powers on vectors are understood component-wise. 
	Let
		$\bm{v} = \frac{1}{2}\bm{c}^2 - A\bm{c}$.
	Then substituting \eqref{eq:theorem_cond_b} into \eqref{eq:theorem_cond_c} and expressing \eqref{eq:theorem_cond_d} in terms of $\bm{v}$ gives
	\begin{align*}
		\bm{b}^T\bm{v} &= \beta_2, \\
		\bm{b}^T\bm{v}^2 &= \beta_2^2. 
	\end{align*}
  Each of 
  these is a strictly convex combination.
  Jensen's inequality (with a strictly convex function, as is the case with the
  square function here) then states $\bm{b}^T \bm{v}^2 \leq (\bm{b}^T
  \bm{v})^2$ with equality if and only if all components of $\bm{v}$ are equal
  \cite[Theorem 12, pg 31]{Bullen:inequalities}.
  Now $v_1 = 0$ for every explicit method so we deduce that $\bm{v}=\bm{0}$.
  That implies the method has stage order two, which is not possible for
  explicit methods \cite{Ruuth2002}.
  This contradiction completes the proof.
\end{proof}

\begin{corollary}\label{cor:no_SSP_5}
    Let $M$ denote an irreducible explicit Runge--Kutta method with $\sspcoef>0$.
    Then $M$ has effective order at most four.
\end{corollary}
\begin{proof}
	This follows immediately from Lemma~\ref{lem:positive_b} and 
	Theorem~\ref{thm:effective_barrier}.
\end{proof}

\begin{remark}
    It is worth noting here an additional result that 
    follows directly from what we have proved.
    Using Theorem~\ref{thm:effective_barrier} and \cite[Theorem~4.1]{dahlquist2006}, 
    it follows that any irreducible explicit Runge--Kutta method with positive radius of
    circle contractivity has effective order at most four.
\end{remark}

        \section{Optimal explicit SSP Runge--Kutta schemes with maximal effective order}\label{sec:optimal_ESSPRK}
In this section, we use the SSP theory and Butcher's theory of effective
order (Sections \ref{sec:SSP} and \ref{sec:Algebraic_RK}) to find
optimal explicit SSP Runge--Kutta schemes with prescribed effective
order and classical order.
According to Corollary~\ref{cor:no_SSP_5}, there are no explicit SSPRK methods of
effective order five, and therefore we need only consider methods with
effective order up to four.

Recall from Section~\ref{sec:Algebraic_RK} that the methods with
an effective order of accuracy involve a main method $M$ as well as starting and
stopping methods $S$ and $S^{-1}$.
In Section~\ref{subsec:starting_stopping} we introduce a novel approach
to construction of starting and stopping methods in order to allow
them to be SSP.

We denote by ESSPRK($s,q,p$) an $s$-stage explicit SSP Runge--Kutta
method of effective order $q$ and classical order $p$.
Also we write SSPRK($s,q$) for an $s$-stage explicit SSP Runge--Kutta
method of order $q$.

\subsection{The main method}\label{subsec:main_method}

Our search is carried out in two
steps, first searching for optimal main methods $M$ and then for
possible corresponding methods $S$ and $S^{-1}$.
For a given number of stages, effective order, and classical order,
our aim is thus to find an optimal main method, meaning one with the 
largest possible SSP coefficient $\sspcoef$.

To find a method ESSPRK($s,q,p$) with Butcher tableau $(A, \bm{b},
\bm{c})$, we consider the optimization problem~\eqref{eq:SSP_opt} 
with $\Phi(K)$ representing the conditions for effective order
$q$ and classical order $p$ (as per Table~\ref{tab:effective_OCs}).
The methods are found through numerical search, using 
\textsc{Matlab}'s optimization toolbox.
Specifically, we use \texttt{fmincon} with a sequential quadratic 
programming approach \cite{Ketcheson2008, Ketcheson/Macdonald/Gottlieb:2009}.
This process does not guarantee a global minimizer, so many searches 
from random initial guesses are performed to help
find methods with the largest possible SSP coefficients.

\subsubsection{Optimal SSP coefficients}\label{subsubsec:optimal_SSP_coeff}
Useful bounds on the optimal SSP coefficient can be obtained 
by considering an important relaxation. 
In the relaxed problem, the method is required to be accurate and strong 
stability preserving only for linear, constant-coefficient initial value problems. 
This leads to a reduced set of order conditions and a relaxed absolute 
monotonicity condition \cite{Kraaijevanger1986,Ketcheson2008,ketcheson2009a}.
We denote the maximal SSP coefficient for linear problems
(maximized over all methods with order $q$ and $s$ stages) by $\clin$.

Let $\sspcoef_{s,q}$ denote the maximal SSP coefficient (relevant to
non-linear problems) over all methods of $s$ stages with order $q$.  Let
$\sspcoef_{s,q,p}$ denote the object of our study, i.e. the maximal SSP
coefficient (relevant to non-linear problems) over all methods of $s$ stages
with effective order $q$ and classical order $p$.
From Remark~\ref{rem:talltrees} and the fact that the ESSPRK($s,q,p$) methods
form a super class of the SSPRK($s,q$) methods, we have
\begin{align} \label{ineq:clin}
        \sspcoef_{s,q} \le \sspcoef_{s,q,p} \le\clin.
\end{align}

The effective SSP coefficients for methods with up to eleven stages are shown in 
Table~\ref{tab:eff_SSP_coeff}.
Recall from Section~\ref{sec:ExRK_barrier} that $q=5$ implies a zero
SSP coefficient and from Section~\ref{sec:Algebraic_RK} that for
$q=1,2$, the class of explicit Runge--Kutta methods with effective order $q$
is the simply the class of explicit Runge--Kutta methods with order $q$.  
Therefore we consider only methods of effective order $q=3$ and $q=4$.
Exact optimal values of $\clin$ are known for many classes of methods; for
example see \cite{Kraaijevanger1986,Ketcheson2008,ketcheson2009a}.
Those results and \eqref{ineq:clin} allow us to determine the optimal value
of $\sspcoef_{s,q,p}$ {\em a priori} for the cases $q=3$ (for any $s$) and
for $q=4,s=10$, since in those cases we have $\sspcoef_{s,q}=\clin$.

\begin{table}
\caption{Effective SSP coefficients $ \ceff = \sspcoef/s$ of the best known
      ESSPRK($s,q,p$) methods.
    		Entries in bold achieve the bound $\clin$ given by the linear SSP coefficient and are therefore optimal. 
    		If no positive $\sspcoef$ can be found, we use ``$-$'' to indicate non-existence. 
    		The optimal fourth-order linear SSP coefficients are $\sspcoef^{\textnormal{lin}}_{4,4}=0.25$,
    		$\sspcoef^{\textnormal{lin}}_{5,4}=0.40$ and $\sspcoef^{\textnormal{lin}}_{6,4}=0.44$.}
    \centering
    \begin{tabular}{ccccccccccccc}
        \toprule
        \multirow{2}{*}{$q$} &
        \multirow{2}{*}{$p$}
            & \multicolumn{11}{c}{stages $s$} \\
            \cmidrule{3-13}
& & $1$ & $2$ & $3$ & $4$ & $5$ & $6$ & $7$ & $8$ & $9$ & $10$ & $11$ \\
        \midrule
        $3$ & $2$ & $-$ & $-$ & $\bf 0.33$ & $\bf 0.50$ & $\bf 0.53$ & $\bf 0.59$ & $\bf 0.61$ & $\bf 0.64$ & $\bf 0.67$ & $\bf 0.68$ & $\bf 0.69$\\
        $4$ & $2$ & $-$ & $-$ & $-$ & $0.22$ & $0.39$ & $\bf 0.44$ & $\bf 0.50$ & $\bf 0.54$ & $\bf 0.57$ & $\bf 0.60$ & $\bf 0.62$ \\
        $4$ & $3$ & $-$ & $-$ & $-$ & $0.19$ & $0.37$ & $0.43$ & $\bf 0.50$ & $\bf 0.54$ & $\bf 0.57$ & $\bf 0.60$ & $\bf 0.62$ \\
        \bottomrule
    \end{tabular}
    \label{tab:eff_SSP_coeff}
\end{table}

\subsubsection{Effective order three methods}\label{subsubsec:3rd_ESSPRK}
Since $\sspcoef_{s,q}=\clin$ for $q=3$, the optimal effective order three methods
have SSP coefficients equal to the corresponding optimal classical order three methods.
In the cases of three and four stages, we are able to determine exact coefficients for
families of optimal methods of effective order three.
\begin{theorem}\label{thm:ESSPRK(3,3,2)}
	A family of optimal three-stage, effective order three SSP Runge--Kutta 
	methods of classical order two, with SSP coefficient $\sspcoef_{3,3,2} = 1$, is given by
    \begin{displaymath}
    		\begin{split}
    			\bm{Y}_1 &= \bm{u}^n, \\
    			\bm{Y}_2 &= \bm{u}^n + \Dt\bm{F}(\bm{Y}_1), \\
    			\bm{Y}_3 &= \bm{u}^n + \gamma\Dt\bm{F}(\bm{Y}_1) + \gamma\Dt\bm{F}(\bm{Y}_2), \\
    			\bm{u}^{n+1} &= \bm{u}^n + \frac{5\gamma-1}{6\gamma}\Dt\bm{F}(\bm{Y}_1) + \frac{1}{6}\Dt\bm{F}(\bm{Y}_2) + \frac{1}{6\gamma}\Dt\bm{F}(\bm{Y}_3),
        \end{split}
    \end{displaymath}
    where $1/4 \leq \gamma \leq 1$ is a free parameter.
\end{theorem}
\begin{theorem}\label{thm:ESSPRK(4,3,2)}
	A family of optimal four-stage, effective order three SSP Runge--Kutta 
	methods of classical order two, with SSP coefficient $\sspcoef_{4,3,2} = 2$ is given by
    \begin{displaymath}
    		\begin{split}
    			\bm{Y}_1 &= \bm{u}^n, \\
    			\bm{Y}_2 &= \bm{u}^n + \frac{1}{2}\Dt\bm{F}(\bm{Y}_1), \\
    			\bm{Y}_3 &= \bm{u}^n + \frac{1}{2}\Dt\bm{F}(\bm{Y}_1) + \frac{1}{2}\Dt\bm{F}(\bm{Y}_2), \\
    			\bm{Y}_4 &= \bm{u}^n + \gamma\Dt\bm{F}(\bm{Y}_1) + \gamma\Dt\bm{F}(\bm{Y}_2) + + \gamma\Dt\bm{F}(\bm{Y}_3), \\
    			\bm{u}^{n+1} &= \bm{u}^n + \frac{8\gamma-1}{12\gamma}\Dt\bm{F}(\bm{Y}_1) + \frac{1}{6}\Dt\bm{F}(\bm{Y}_2) + \frac{1}{6}\Dt\bm{F}(\bm{Y}_3) + \frac{1}{12\gamma}\Dt\bm{F}(\bm{Y}_4),
        \end{split}
    \end{displaymath}
    where $ 1/6 \leq \gamma \leq 1/2 $ is a free parameter.
\end{theorem}
\begin{proof}
	In either theorem, feasibility can be verified by direct calculation of the 
	conditions in problem~\eqref{eq:SSP_opt}. Optimality follows because 
	$\sspcoef_{s,3,2} = \sspcoef^{\textnormal{lin}}_{s,3}$.
\end{proof}

Theorem~\ref{thm:ESSPRK(3,3,2)} gives a \emph{family} of three-stage 
methods. 
The particular value of $\gamma = 1/4$ corresponds to the classical
Shu--Osher SSPRK($3,3$) method \cite{Gottlieb/Shu:1998}.
Similarly, in Theorem~\ref{thm:ESSPRK(4,3,2)} the particular value of 
$\gamma = 1/6$ corresponds to the usual SSPRK($4,3$) method.
It seems possible that for each number of stages, the 
ESSPRK($s, 3, 2$) methods may form a family in which an optimal 
SSPRK($s$, $3$) method is a particular member. 

\subsubsection{Effective order four methods}\label{subsubsec:4th_ESSPRK}
The ESSPRK($s,4,p$) methods can have classical order $p=2$ or $3$.
In either case, for stages $7 \le s \le 11$ the methods found are
optimal because the SSP coefficient attains the upper bound of
$\clin$.
For fewer stages, the new methods still have SSP coefficients up to
30\% larger than that of explicit SSPRK($s,q$) methods.
In the particular case of four-stage methods we have the following:
\begin{remark}
	In contrast with the non-existence of an SSPRK(4,\,4) method 
	\cite{Gottlieb/Shu:1998,Ruuth2002}, 
	we are able to find ESSPRK(4,\,4,\,2) and ESSPRK(4,\,4,\,3) methods.
	The coefficients of these methods are found in
	Tables~\ref{tab:ESSPRK(4,4,2)_scheme}
	and~\ref{tab:ESSPRK(4,4,3)_scheme}.
\end{remark}

Additionally, we find two families of methods with effective order four, 
for which $\ceff$ asymptotically approaches unity.
The families consist of second order methods with $s = n^2+1$ stages and 
SSP coefficient $\sspcoef_{s,4,2} = n^2-n$.
They are optimal since $\sspcoef_{s,4,2}  = \sspcoef^{\textnormal{lin}}_{s,4}$ 
\cite[Theorem~5.2(c)]{Kraaijevanger1986}.
It is convenient to express the coefficients in the modified Shu--Osher form
\cite{Gottlieb2011a}
\begin{align*}
	\bm{Y}_i &= v_i\bm{u}^n + \sum_{j=1}^{i-1}\bigl(\alpha_{ij}\bm{Y}_j + \Dt\beta_{ij}\bm{F}(\bm{Y}_j)\bigr), \; 1 \leq i \leq s+1 \\
	\bm{u}^{n+1} &= \bm{Y}_{s+1},
\end{align*}
because of the sparsity of the matrices $\alpha, \beta \in \mathbb{R}^{(s+1)\times s}$
and vector $\bm{v} \in \mathbb{R}^s$.
For $n \geq 3$ the non-zero elements are given by
\begin{align*}
	v_1 &= 1, \quad\quad v_{n^2+2} = \frac{2}{(n^2+1)\bigl((n-1)^2+1\bigr)}, \\
	\alpha_{n^2-2n+4,(n-2)^2} &= \frac{n^2-1 \pm \sqrt{n^3-3n^2+n+1}}{4n^2-6n+2}, \\ 		
	\alpha_{n^2+2,n^2+1} &= \frac{n(n-1)^2}{(2n-1)(n^2+1)(1-\alpha_{n^2-2n+4,(n-2)^2})}, \\
	\alpha_{n^2+2,n^2-2n+2} &= 1 - v_{n^2+2} - \alpha_{n^2+2,n^2+1}, \\
	\alpha_{i+1,i} &= \begin{cases} 
								1 - \alpha_{i+1,(n-2)^2}, & i = n^2-2n+3 \\
								1,  &\mbox{otherwise,}
							\end{cases}
\end{align*}
where $ 1 \leq i \leq n^2$ and 
\begin{align*}
	\beta_{i,j} & = \frac{\alpha_{i,j}}{n^2-n}, \quad 1 \leq i \leq n^2+2, \;\; 1 \leq j \leq n^2+1.
\end{align*}
In \cite[\S~6.2.2]{Gottlieb2011a}, a similar pattern was found
for SSPRK($s,3$) methods.

\subsection{Starting and stopping methods}\label{subsec:starting_stopping}
Provided an ESSPRK($s,q,p$) scheme that can be used as the main 
method $M$, we want to find perturbation methods $S$ and $S^{-1}$ such that the 
Runge--Kutta scheme $S^{-1}MS$ attains classical order $q$, equal to the 
effective order of method $M$.
We also want the resulting overall process to be SSP.
However at least one of the $S$ and $S^{-1}$ methods is not SSP:
if $\beta_1 = 0$ then $\sum_i b_i = 0$ implies the presence of at
least one negative weight and thus neither scheme can be SSP.
Even if we consider methods with $\beta_1 \neq 0$, one of $S$ or
$S^{-1}$ must step backwards and thus that method cannot be SSP
(unless we consider the downwind operator
\cite{Ruuth2004,Gottlieb/Ruuth:SSPfastdownwind,Ketcheson:2011:downwind}).

In order to overcome this problem and achieve ``bona fide''
SSPRK methods with an effective order of accuracy, we need to choose different starting and stopping methods. 
We consider methods $R$ and $T$ which each take a positive step such that 
$R \equival{q} MS$ and $T \equival{q} S^{-1}M$.
That is, the order conditions of $R$ and $T$ must match those of
$MS$ and $S^{-1}M$, respectively, up to order $q$.
This gives a new $TM^{n-2}R$ scheme which is equivalent up to order $q$ 
to the $S^{-1}M^nS$ scheme and attains classical order $q$.
Each starting and stopping procedure now takes a positive step forward
in time.

To derive order conditions for the $R$ and $T$ methods, consider their
corresponding functions in group $G$ to be $\rho$ and $\tau$
respectively.
Then the equivalence is expressed as
\begin{equation} \label{eq:R_T_OCs}
    \rho(t) = (\beta\alpha)(t) \text{ and } \tau(t) = (\alpha\beta^{-1})(t), \quad \text{for all 
    trees $t$ with $r(t) \leq q$.}
\end{equation}
Rewriting the second condition in \eqref{eq:R_T_OCs} as 
$(\tau\beta)(t) = \alpha(t)$, the order conditions for the starting and stopping 
methods can be determined by the usual product formula and are given in Table~\ref{tab:rho_tau_OCs}.
These conditions could be constructed more generally but here we have
assumed $\beta_1=0$ (see Section~\ref{subsubsec:Main_starting_conditions}); this
will be sufficient for constructing SSP starting and stopping
conditions.

\begin{table}
  	\caption{Order conditions on $\rho$ and $\tau$ up to effective order four for starting
  		and stopping methods $R$ and $T$, respectively.
  		The upper block represents the effective order three conditions.
         As in Table~\ref{tab:effective_OCs_on_alpha} and Table~\ref{tab:effective_OCs} we assume 
         $\beta_1 = 0$.}
	\centering
	\begin{tabular}{lcl}
		\toprule
    		$\rho(t) = (\beta\alpha)(t)$ & & $\tau(t) = (\alpha\beta^{-1})(t)$ \\
    		\midrule
    		 $\rho_1 = \alpha_1$ & & $\tau_1 = \alpha_1$ \\
    		$\rho_2 = \alpha_2 + \beta_2$ & & $\tau_2 = \alpha_2 - \beta_2$ \\
    		$\rho_3 = \alpha_3 + \beta_3$ & & $\tau_3 = \alpha_3 - 2\alpha_1\beta_2 - \beta_3$ \\
    		$\rho_4 = \alpha_4 + \alpha_1\beta_2 + \beta_4$ & & $\tau_4 = \alpha_4 - \alpha_1\beta_2 - \beta_4$ \\
                \mydashrule
		$\rho_5 = \alpha_5 + \beta_5$ & & $\tau_5 = \alpha_5 - 3\alpha_1^2\beta_2 - 3\alpha_1\beta_3 - \beta_5$ \\
		$\rho_6 = \alpha_6 + \alpha_2\beta_2 + \beta_6$ & & $\tau_6 = \alpha_6 - (\alpha_1^2 + \alpha_2 -\beta_2)\beta_2 -\alpha_1\beta_3 - \alpha_1\beta_4 - \beta_6$ \\
		$\rho_7 = \alpha_7 + \alpha_1\beta_3 + \beta_7$ & & $\tau_7 = \alpha_7 - 2\alpha_1\beta_4 - \alpha_1^2\beta_2 - \beta_7$ \\
		$\rho_8 = \alpha_8 + \alpha_1\beta_4 + \alpha_2\beta_2 + \beta_8$ & & $\tau_8 = \alpha_8 - \alpha_1\beta_4 - \alpha_2\beta_2 + \beta_2^2 -  \beta_8$
                \\
                \bottomrule
  	\end{tabular}
  	\label{tab:rho_tau_OCs}
\end{table}

\subsubsection{Optimizing the starting and stopping methods}\label{subsubsec:opt_methods}
It turns out that the order conditions from \eqref{eq:R_T_OCs} do not
contradict the SSP requirements.
We can thus find methods $R$ and $T$ using the optimization procedure
described in Section~\ref{subsec:Optimal_SSPRK} with the order conditions 
given by Table~\ref{tab:rho_tau_OCs} for $\Phi(K)$ in \eqref{eq:SSP_opt}.

The values of $\alpha_i$ are determined by the main method $M$.
Also note that for effective order $q$, the algebraic expressions on
$\beta$ up to order $q-1$ are already found by the optimization procedure of 
the main method (see Table~\ref{tab:effective_OCs}). 
However, the values of the order $q$ elementary weights on $\beta$ are not 
known; these are $\beta_3$ and $\beta_4$ for effective order three and
$\beta_5$, $\beta_6$, $\beta_7$ and $\beta_8$ for effective order four.
From Table~\ref{tab:rho_tau_OCs}, we see that both the $R$ and $T$
methods depend on these parameters.
Our approach is to optimize for both methods at once: we solve a
modified version of the optimization problem \eqref{eq:SSP_opt} where
we simultaneously maximize both SSP coefficients subject to the
constraints given in \eqref{eq:R_T_OCs} and conditions on $\beta$ given by 
Table~\ref{tab:effective_OCs}. 
The unknown elementary weights on $\beta$ are used as free parameters.
In practice, we maximize the objective function $\min(r_1,r_2)$, where $r_1$ 
and $r_2$ are the radii of absolute monotonicity of the methods $R$ and $T$.

We were able to construct starting and stopping schemes for each main 
method, with an SSP coefficient at least as large as that of the main method.
This allows the usage of a uniform time-step $\Dt \leq \sspcoef\DtFE$, 
where $\sspcoef$ is the SSP coefficient of the main method.
The additional computational cost
of the starting and stopping methods is minimal:
for methods $R$ and $T$ associated with an $s$-stage main method,
at most $s + 1$ and $s$ stages, respectively, appear to be required.
Tables \ref{tab:ESSPRK(4,4,2)_scheme} and \ref{tab:ESSPRK(4,4,3)_scheme} 
show the coefficients of the schemes where the main 
method is ESSPRK($4,4,2$) and ESSPRK($4,4,3$), respectively. 

It is important to note that in practice, if accurate values are needed at 
any time other than the final time, the computation must invoke the 
stopping method to obtain them.  Furthermore, changing step-size would require first applying the stopping method with the old step-size and then applying the starting method with the new step-size.


\begin{table}
    \caption{ESSPRK(4,4,2): an effective order four SSPRK method with
      four stages and classical order two with its associated starting
      and stopping methods.}
    \setlength{\tabcolsep}{2pt}
    \footnotesize
    \subfloat[Main method $M$, ESSPRK($4,4,2$) \label{ESSPRK(4,4,2)_scheme_a}]{
        \begin{tabular}{c | c c c c}
             $0$ & & & & \\
             $0.730429885783319$ & $0.730429885783319$ & & & \\
             $0.644964638145795$ & $0.251830917810810$ & $0.393133720334985$ & & \\
             $1.000000000000000$ & $0.141062771617064$ & $0.220213358584678$ & $0.638723869798257$ & \\
             \hline
             & $0.384422161080494$ & $0.261154113377550$ & $0.127250689937518$ & $0.227173035604438$
        \end{tabular}
    }\\
    \subfloat[Starting method $R$ \label{ESSPRK(4,4,2)_scheme_b}]{
        \begin{tabular}{c | c c c c c}
			$0$ & & & & & \\
			$0.545722177514735$ & $0.545722177514735$ & & & & \\
			$0.842931687441527$ & $0.366499989048164$ & $0.476431698393363$ & & & \\
			$0.574760809487828$ & $0.135697968350722$ & $0.176400587890242$ & $0.262662253246864$ & & \\
			$0.980872743236632$ & $0.103648417776838$ & $0.134737771331049$ & $0.200625899485633$ & $0.541860654643112$ & \\		
             \hline
             & $0.233699169638954$ & $0.294263351266422$ & $0.065226988215286$ & $0.176168374199685$ & $0.230642116679654$ 
        \end{tabular}        
    }\\
    \subfloat[Stopping method $T$ \label{ESSPRK(4,4,2)_scheme_c}]{
        \begin{tabular}{c | c c c c}
			$0$ & & & & \\
			$0.509877496215340$ & $0.509877496215340$ & & & \\
			$0.435774135529007$ & $0.182230305923759$ & $0.253543829605247$ & & \\
			$0.933203341300203$ & $0.148498121305090$ & $0.206610981494095$ & $0.578094238501017$ & \\
             \hline            
             & $0.307865440399752$ & $0.171863794704750$ & $0.233603236964822$ & $0.286667527930676$
        \end{tabular}
    }
    \label{tab:ESSPRK(4,4,2)_scheme}
\end{table}

\begin{table}
    \caption{ESSPRK(4,4,3): an effective order four SSPRK method with
      four stages and classical order three with its associated starting
      and stopping methods.}
    \setlength{\tabcolsep}{2pt}
    \footnotesize
    \subfloat[Main method $M$, ESSPRK($4,4,3$) \label{ESSPRK(4,4,3)_scheme_a}]{
        \begin{tabular}{c | c c c c}
             $0$ & & & & \\
             $0.601245068769724$ & $0.601245068769724$ & & & \\
             $0.436888719886063$ & $0.139346829159954$ & $0.297541890726109$ & & \\
             $0.747760163757110$ & $0.060555450075478$ & $0.129301708677891$ & $0.557903005003740$ & \\
             \hline
             & $0.220532078662434$ & $0.180572397883936$ & $0.181420582644840$ & $0.417474940808790$
        \end{tabular}  
    }\\
    \subfloat[Starting method $R$ \label{ESSPRK(4,4,3)_scheme_b}]{
        \begin{tabular}{c | c c c c c}
			$0$ &  & & & & \\
			$0.438463764036947$ & $0.438463764036947$ & & & & \\
			$0.639336395725557$ & $0.213665532574654$ & $0.425670863150903$ & & & \\
			$0.434353425654020$ &$0.061345094040860$ & $0.122213530726218$ & $0.250794800886942$ & & \\
			$0.843416464962307$ & $0.039559973266996$ & $0.078812561688700$ & $0.161731525131914$ & $0.563312404874697$ & \\			
             \hline
             & $0.154373542967849$ & $0.307547588471376$ & $0.054439037790856$ & $0.189611674483496$ & $0.294028156286422$
        \end{tabular}
    }\\
    \subfloat[Stopping method $T$ \label{ESSPRK(4,4,3)_scheme_c}]{
        \begin{tabular}{c | c c c c}
			$0$ & & & & \\
			$0.556337718891090$ & $0.556337718891090$ & & & \\
			$0.428870688216872$ & $0.166867537553458$ & $0.262003150663414$ & & \\
			$0.815008947642716$ & $0.104422177204659$ & $0.163956032598547$ & $0.546630737839510$ & \\		
             \hline
             & $0.203508169408374$ & $0.096469758967330$ & $0.321630956102914$ & $0.378391115521382$
        \end{tabular}
    }
    \label{tab:ESSPRK(4,4,3)_scheme}
\end{table}

        \section{Numerical experiments}\label{sec:numerics}
Having constructed strong stability preserving $TM^{n-2}R$ schemes in the previous
section, we now numerically verify their properties.
Specifically, we use a convergence study to show that the procedure
attains order of accuracy $q$, the effective order of $M$.
We also demonstrate on Burgers' equation that the
SSP coefficient accurately measures the maximal time-step for which the
methods are strong stability preserving.

\subsection{Convergence study}\label{subsec:convergence}
We consider the van der Pol system \cite{Hairer1987_book}
\begin{equation}\label{eq:conv_eq}
	\begin{split}
    		u_1'(t) &= u_2(t), \\
                u_2'(t) &= \mu \bigl(1 - u_1^2(t)\bigr)u_2(t) - u_1(t),
    \end{split}
\end{equation}
over the time interval $t \in [0, 50]$ with $\mu = 2$ and initial
values $u_1(0) = 2$ and $u_2(0) = 1$.
The reference solution for the convergence study is calculated by \textsc{Matlab}'s 
\texttt{ode45} solver with relative and absolute tolerances set to $10^{-13}$.

We solve the initial value problem \eqref{eq:conv_eq}
using SSP $TM^{n-2}R$ schemes.
The solution is computed using $n = 100 \cdot 2^{k}$ time steps for
$k = 2, \dots, 7$.
The error at $t = 50$ with respect to time-step is shown in 
Figure~\ref{fig:conv_study} on a logarithmic scale.
\begin{figure}
	\centering
     \subfloat[ESSPRK($s,3,2$)]{\label{fig:conv_study_a}
     \includegraphics[width=0.45\textwidth]{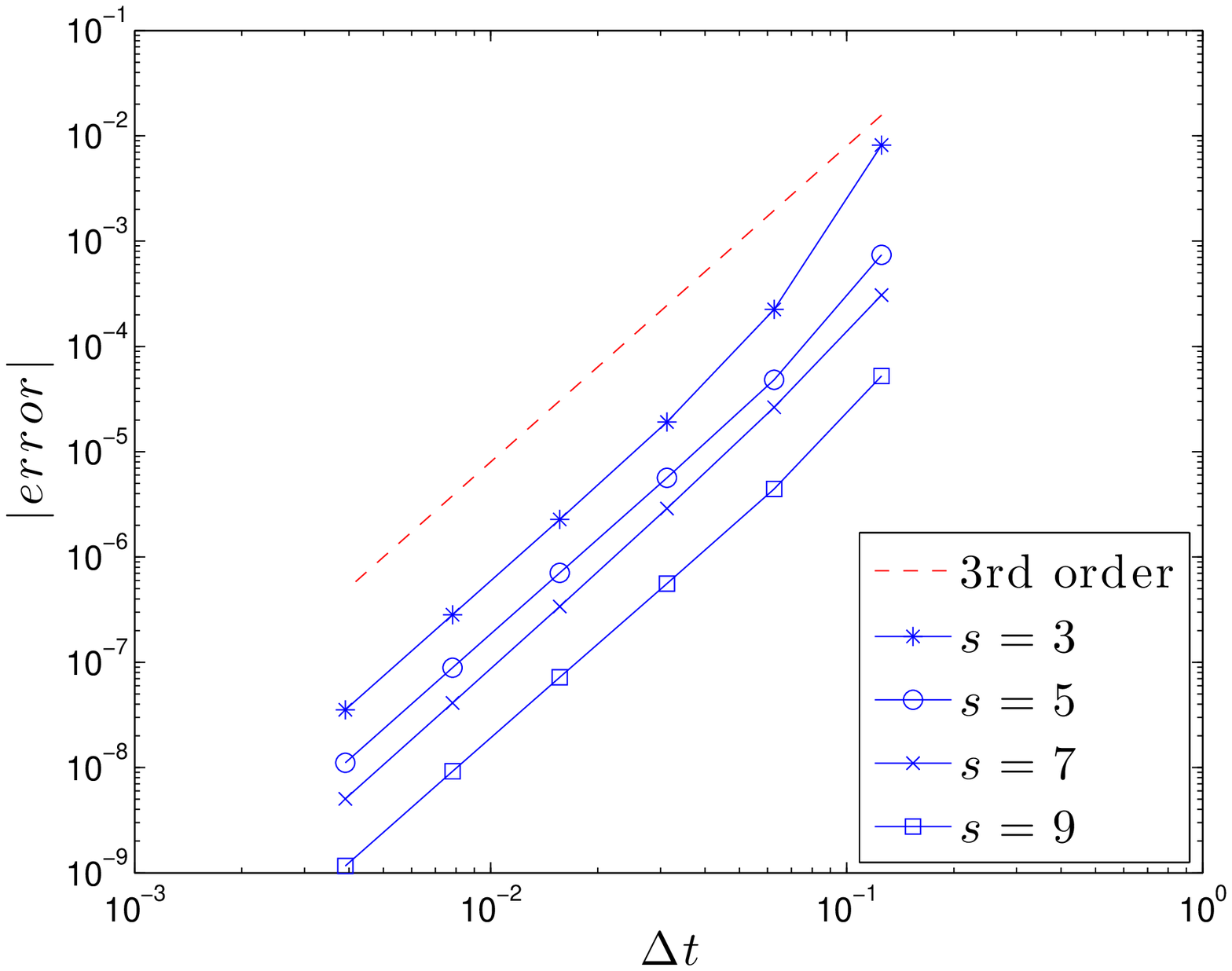}}
   \qquad
     \subfloat[ESSPRK($s,4,2$)]{\label{fig:conv_study_b}
    \includegraphics[width=0.45\textwidth]{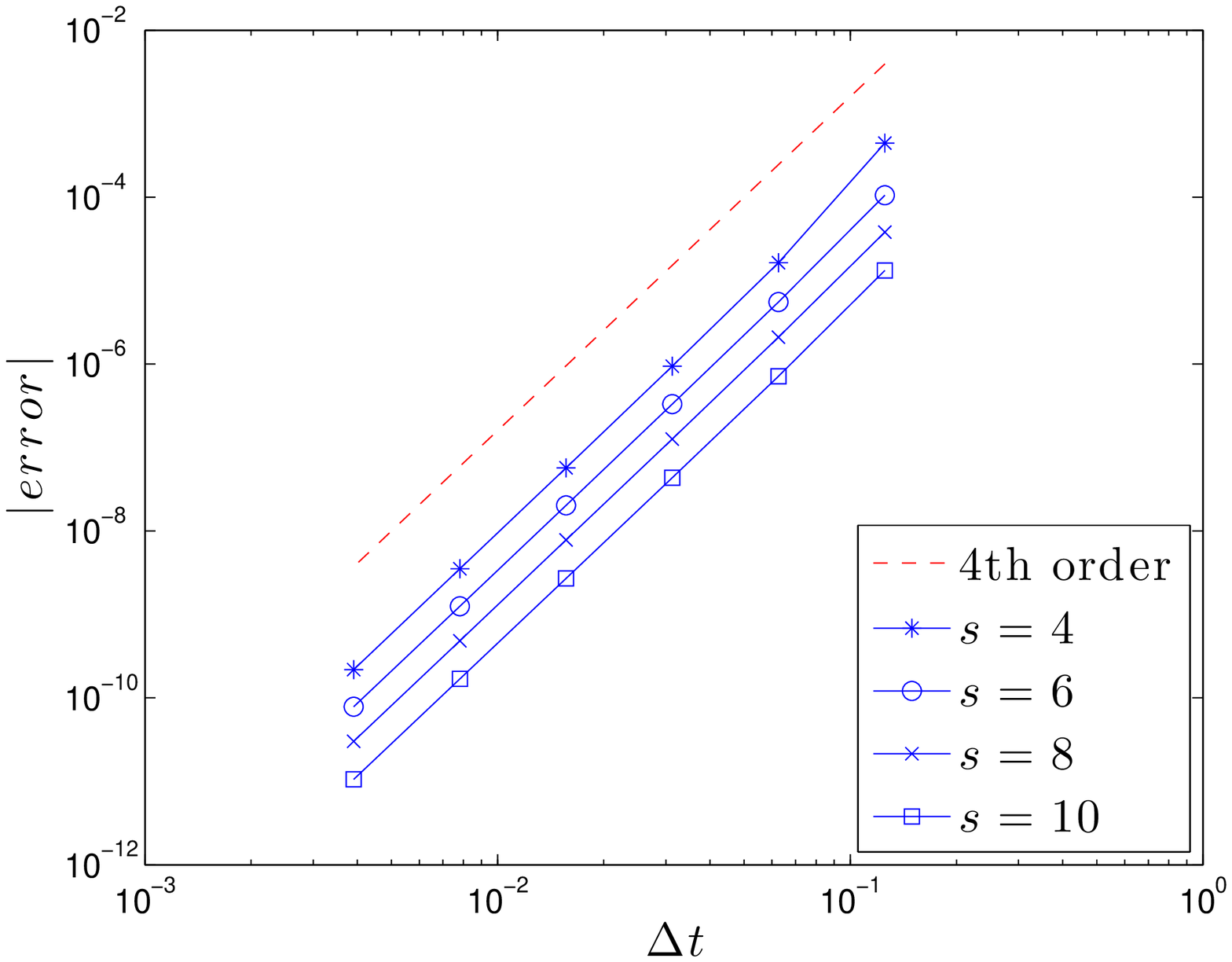}}
    \caption{Convergence study for the van der Pol system using $TM^{n-2}R$ 
    Runge--Kutta schemes when (a) $M$
    is an ESSPRK($s,3,2$) method and (b) $ M $ is an ESSPRK($s,4,2$) method.}
    \label{fig:conv_study}
\end{figure}
The convergence study is performed for $TM^{n-2}R$ schemes with
various number of stages $s$ and the results show that the schemes
attain an order of accuracy equal to the effective order of their main
method $M$.
It is important in doing this sort of convergence study that the
effective order of accuracy can only be obtained after the stopping method is
applied.
Intermediate steps will typically only be order $p$ accurate (the classical
order of the main method).
Finally, we note that the methods with more stages generally exhibit smaller errors
(for a given step size).

\subsection{Burgers' equation}\label{subsubsec:burgers}
The inviscid Burgers' equation consists of the scalar hyperbolic conservation law
\begin{align}\label{eq:HCL}
    U_{t} + f(U)_{x} = 0,
\end{align}
with flux function $f(U) = \frac{1}{2}U^{2}$. 
We consider initial data
$U(0,x)  = \frac{1}{2} - \frac{1}{4}\sin{\pi x}$,
on a periodic domain $x \in [0,2)$.
The solution advances to the right where it eventually exhibits a shock. 
We perform a semi-discretization
using an upwind approximation to obtain the system of ODEs
\begin{align*}\label{eq:burgers_flux}
	\frac{\textrm{d}}{\textrm{d} t} u_i = -\frac{f(u_{i}) - f(u_{i-1})}{\Delta x}.
\end{align*}
This spatial discretization is total-variation-diminishing (TVD) when
coupled with the forward Euler method under the restriction
\cite{Laney:1998}  
$$\Dt \leq {\Dt}_{\text{FE}} = \Delta x / \|U(0,x)\|_{\infty}.$$
Recall that a time discretization with SSP
coefficient $\sspcoef$ will give a TVD solution for $\Dt \leq
\sspcoef {\Dt}_{\text{FE}}$.

Burgers' equation was solved using an SSP $TM^{n-2}R$ scheme with time-step
restriction $\Dt = \sigma{\Dt}_{\text{FE}}$, where $\sigma$ indicates the size 
of the time step. 
We integrate to roughly time $t_\text{f} = 1.62$ with $200$ points in space.
Figure~\ref{fig:burgers_cont} shows that if $\sigma$ is chosen less than the SSP
coefficient of the main method, then no oscillations are observed. 
If this stability limit is violated, then oscillations may appear,
as shown in Figure~\ref{fig:burgers_cont_b}.
We measure these oscillations by computing the total variation of the
numerical solution.


\begin{figure}
    \centering
    \subfloat[$\sigma = 0.88$]{\label{fig:burgers_cont_a}
      \includegraphics[width=0.45\textwidth]{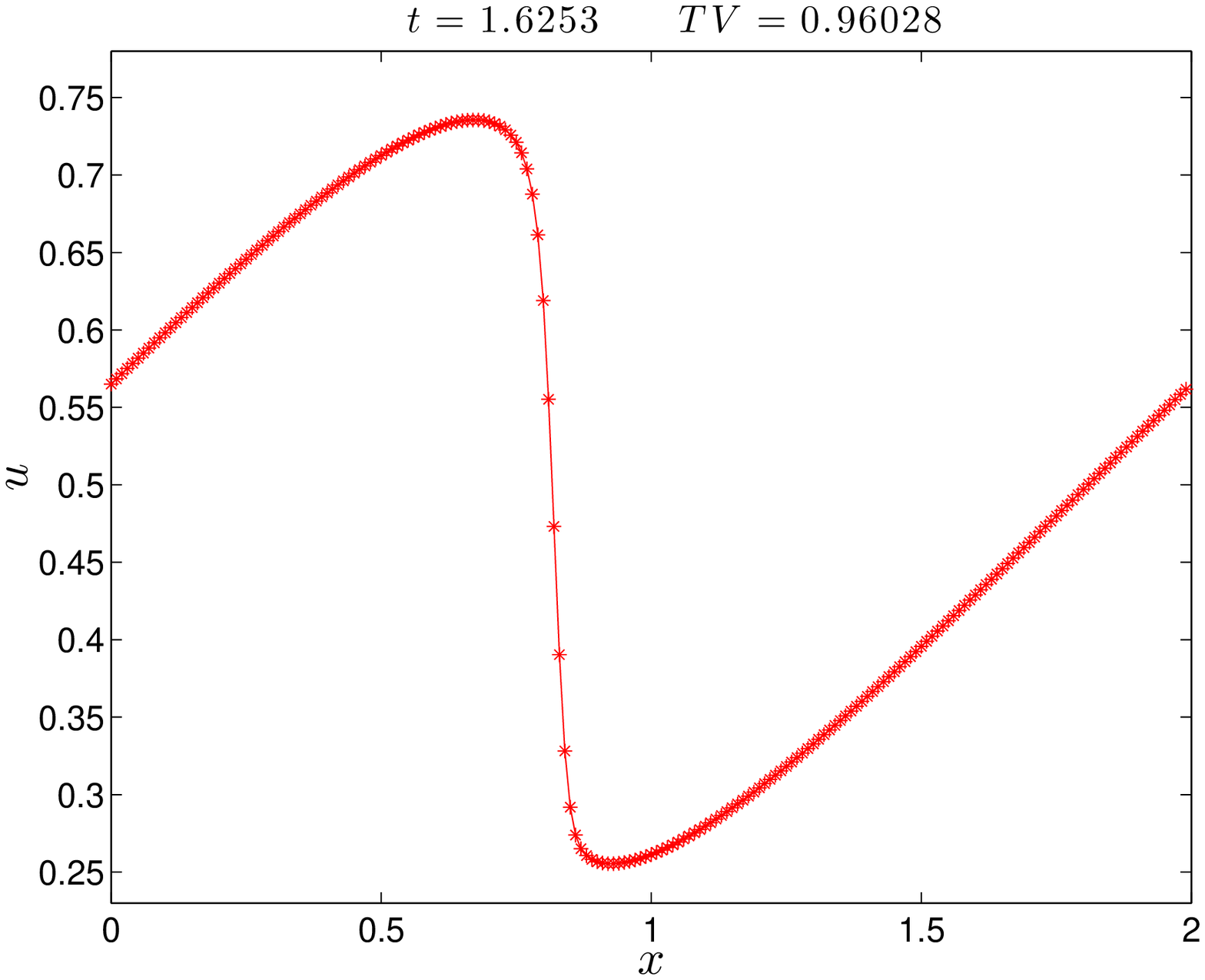}}
    \qquad
    \subfloat[$\sigma = 1.60$]{\label{fig:burgers_cont_b}
      \includegraphics[width=0.45\textwidth]{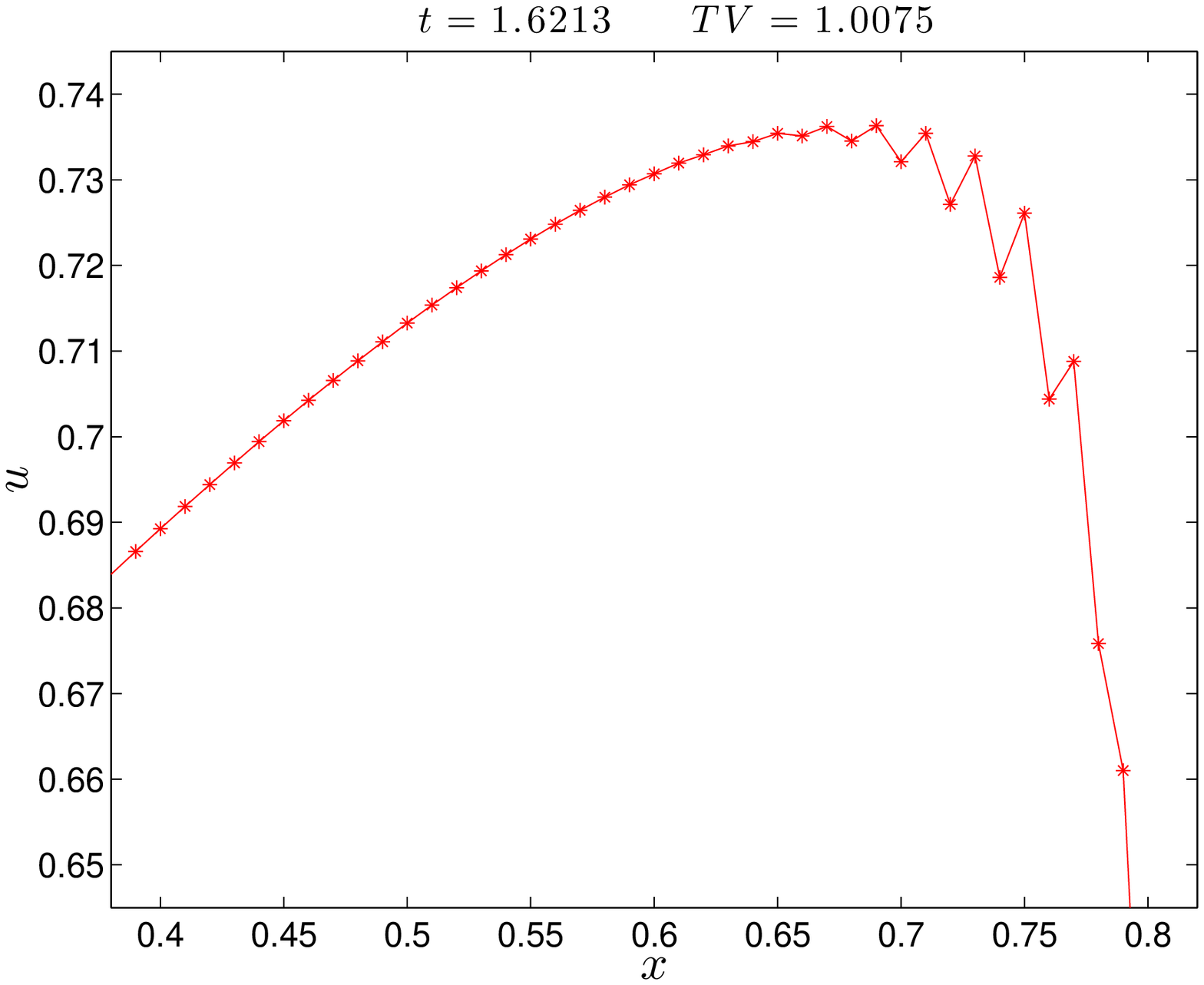}}
    \caption{Solution of Burgers' equation at the final time with continuous initial data, using a
    $TM^{n-2}R$ scheme, where $M$ is the optimal ESSPRK($4,4,2$).
    The time-step used is $\Dt = \sigma{\Dt}_{\text{FE}}$:
    in (a) we take $\sigma$ as the SSP coefficient
    $\sigma = \sspcoef = s \times 0.22 = 0.88$
    (see Table~\ref{tab:eff_SSP_coeff}) and no oscillations are observed.
    However, in (b) we take $\sigma > \sspcoef$ and we observe oscillations
    (note (b) is magnified to show these).
    Here $TV$ denotes the discrete total variation of the solution at the final time:
    a value greater than 1 (the total variation of the initial condition)
    indicates a violation of the TVD condition.
    }
    \label{fig:burgers_cont}
\end{figure}

We also consider Burgers' equation with a discontinuous
square wave initial condition
\begin{equation}\label{eq:burgers_discont_IC}
    U(0,x)  = \left\{
                \begin{array}{ll}
                  1, & \hbox{$0.5 \leq x \leq 1.5$} \\
                  0, & \hbox{otherwise.}
                \end{array}
              \right.
\end{equation}
The solution consists of a rarefaction (i.e., an expansion fan) and a
moving shock.
Again we use $200$ points in space and we compute the solution until
roughly time $t_\text{f} = 0.6$, using a time-step $\Dt = \sigma{\Dt}_{\text{FE}}$.
Figure~\ref{fig:burgers_discont} shows the result of solving the
discontinuous problem using an SSP $TM^{n-2}R$ scheme, where $M$ is an
ESSPRK($5,4,2$) method with SSP coefficient $\sspcoef = 1.95$.
In this case, $\sigma = 1.98$ appears to be the largest value
for which the total variation is monotonically decreasing during the 
calculation.
This is only $2\%$ larger than the value of the SSP coefficient.
Figure~\ref{fig:burgers_discont_b} shows part of the solution exhibiting 
oscillations when $\sigma$ is larger than the SSP coefficient.

\begin{figure}
    \centering
    \subfloat[$\sigma = 1.95$]{\label{fig:burgers_discont_a}
      \includegraphics[width=0.45\textwidth]{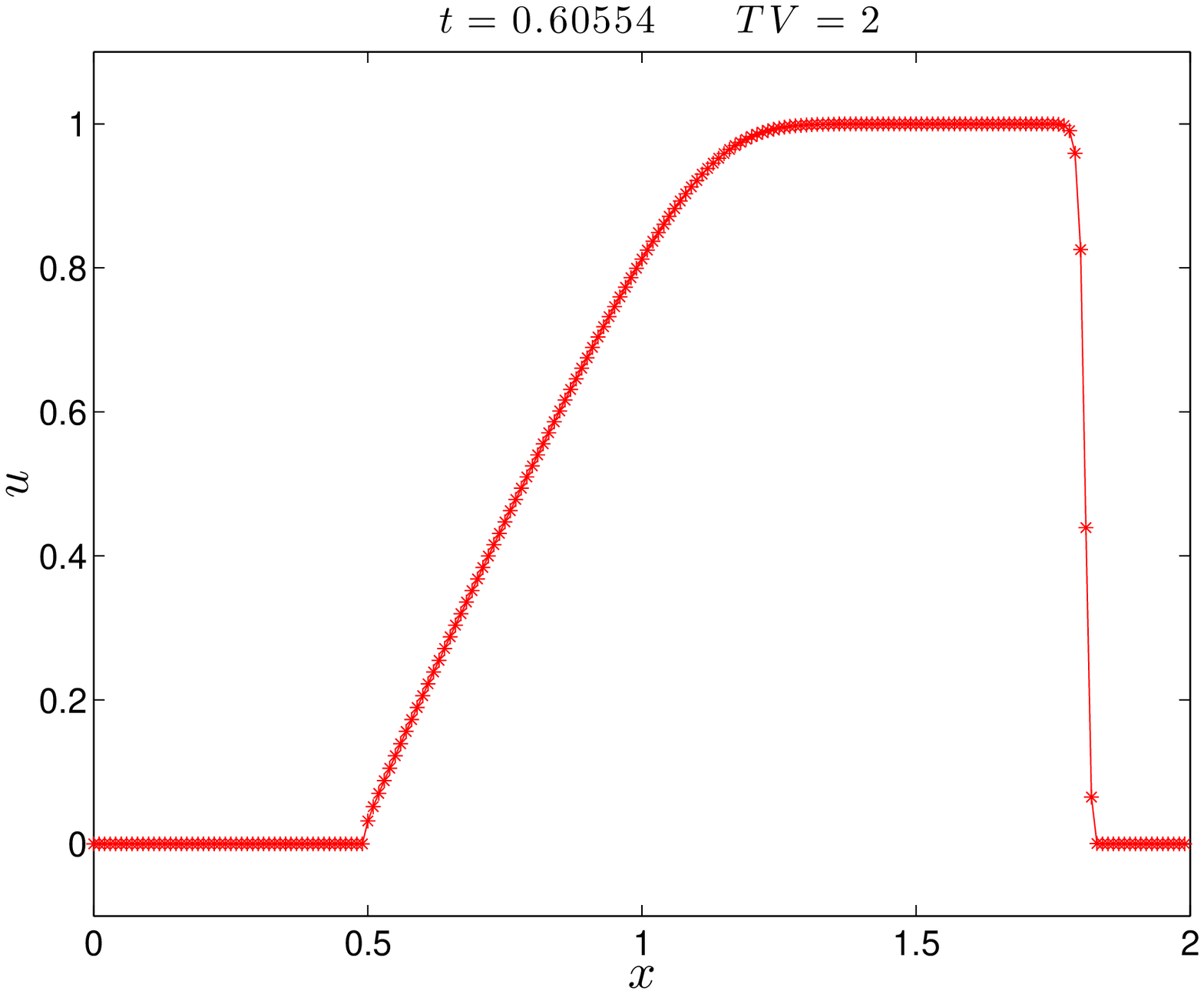}}
    \qquad
    \subfloat[$\sigma = 2.15$]{\label{fig:burgers_discont_b}
      \includegraphics[width=0.45\textwidth]{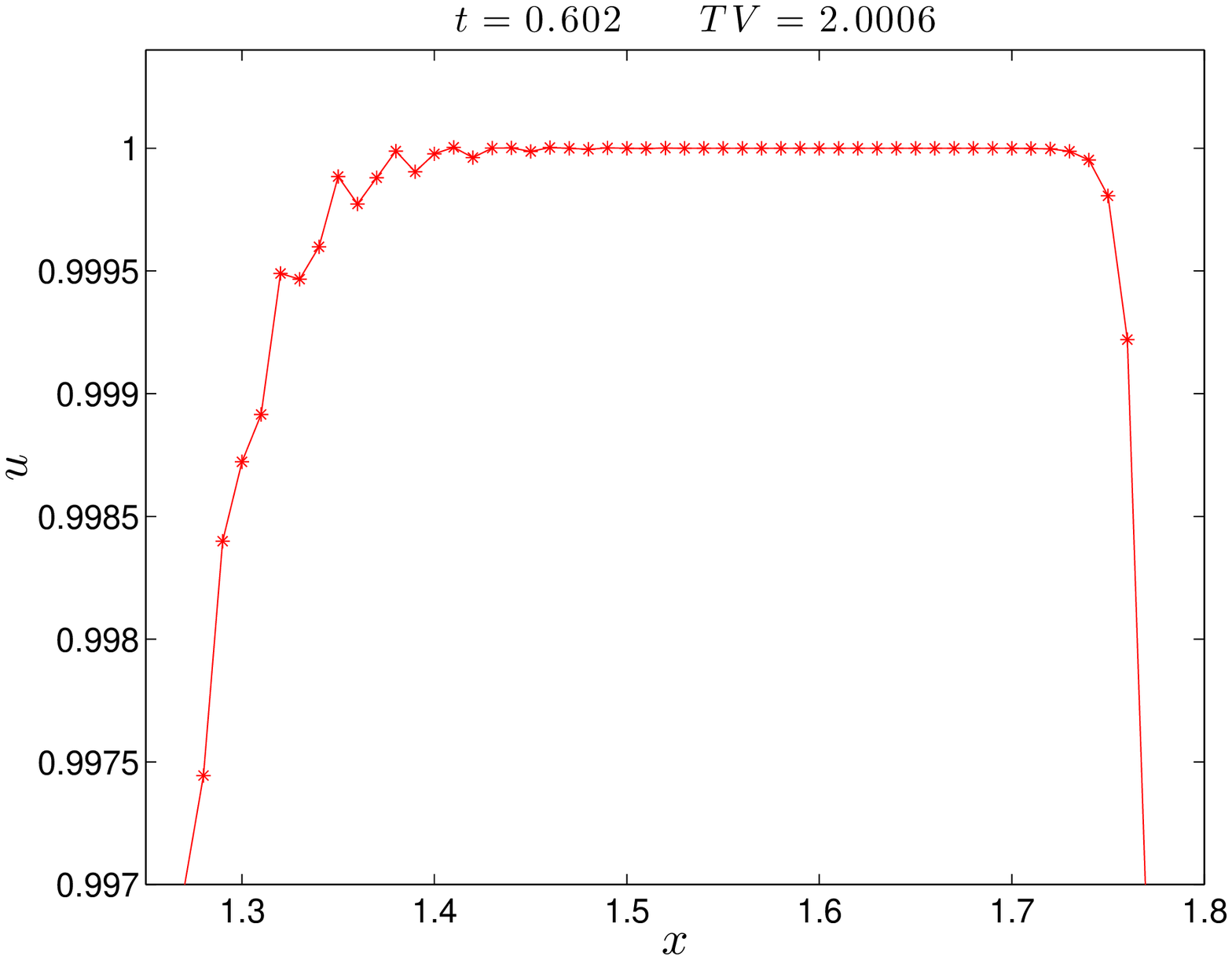}}
    \caption{Solution of Burgers' equation at the final time with discontinuous initial data, using a
    $TM^{n-2}R$ scheme, where $M$ is ESSPRK($5,4,2$) method.
    The time-step used is $\Dt = \sigma{\Dt}_{\text{FE}}$:
    in (a) we take $\sigma$ as the SSP coefficient
    $\sigma = \sspcoef = s \times 0.39 = 1.95$
    and no oscillations are observed.
    However, in (b) we take $\sigma > \sspcoef$ and, when magnified,
    we observe oscillations.
    Here $TV$ denotes the total variation semi-norm of the solution at the final time:
    a value greater than 2 indicates an increase in total variation.
    }
    \label{fig:burgers_discont}
\end{figure}

For various schemes, Table~\ref{tab:observed_SSP_coeff} shows the
maximum observed values of $\sigma$ for which the numerical solution
is total variation decreasing for the entire computation.  With the
exception of the four-stage effective order four methods, we note good
agreement between these experimental values and the SSP coefficients
predicted by the theory.

\begin{table}
    \caption{Maximum observed coefficients exhibiting the TVD property on 
    the Burgers' equation example with discontinuous data \eqref{eq:burgers_discont_IC}.
    The numbers in parenthesis indicate the increase relative to the corresponding SSP
    coefficients.}
    \centering
    \begin{tabular}{>{\hspace*{-4pt}}c@{\hspace{3pt}}c@{\hspace{3pt}}c@{\hspace{2pt}} c@{\hspace{2pt}} c@{\hspace{2pt}} c@{\hspace{2pt}} c@{\hspace{2pt}} c@{\hspace{2pt}} c@{\hspace{2pt}} c@{\hspace{2pt}} c@{\hspace{2pt}}}
        \toprule
        \multirow{2}{*}{$q$} &
        \multirow{2}{*}{\;$p$\;}
               &   \multicolumn{9}{c}{stages $s$} \\
            \cmidrule{3-11}
        &      &   $3$ & $4$ & $5$ & $6$ & $7$ & $8$ & $9$ & $10$ & $11$ \\
        \midrule
        $3$ & $2$ & \small$1.04(4\%)$ & \small$2.00(0\%)$ & \small$2.65(0\%)$ & \small$3.52(0\%)$ & \small$4.29 (0\%)$ & \small$5.11(0\%)$ & \small$6.00(0\%)$ & \small$6.79(0\%)$ & \small$7.63(0\%)$\\
        $4$ & $2$ & \small$-$ & \small$1.07(22\%)$ & \small$1.98(2\%)$ & \small$2.69(2\%)$ & \small$3.56(1\%)$ & \small$4.33(1\%)$ & \small$5.16(1\%)$ & \small$6.05(1\%)$ & \small$6.84(1\%)$ \\
        $4$ & $3$ & \small$-$ & \small$1.05(35\%)$ & \small$1.89(3\%)$ & \small$2.63(2\%)$ & \small$3.53(1\%)$ & \small$4.31(1\%)$ & \small$5.16(1\%)$ & \small$6.04(1\%)$ & \small$6.85(1\%)$ \\
        \bottomrule
    \end{tabular}
    \label{tab:observed_SSP_coeff}
\end{table}
    
We also note the necessity of our modified starting and stopping methods in the $RM^{n-2}T$ approach: in this
example if we use the original approach of $S$ and $S^{-1}$,
the solution exhibits oscillations immediately following the
application of the starting perturbation method $S$.

        \section{Conclusions}\label{sec:Conclusion}
We use the theory of strong stability preserving time discretizations
with Butcher's algebraic interpretation of order to construct
explicit SSP Runge--Kutta schemes with an effective order of accuracy.
These methods, when accompanied by starting and stopping
methods, attain an order of accuracy higher than their (classical) order.
We propose a new choice of starting and stopping methods to allow the
overall procedure to be SSP.
We prove that explicit Runge--Kutta methods with strictly positive 
weights have at most effective order four. 
This extends the barrier already known in the case of classical order
explicit SSPRK methods.

SSP Runge--Kutta methods of effective order three and four
are constructed by numerical optimization.
Most of the methods found are optimal because they achieve
the upper bound on the SSP coefficient known from linear
problems.
Also, despite the non-existence of four-stage, order four explicit SSPRK methods, 
we find effective order four methods with four stages (of classical 
order two and three). 
We perform numerical tests which confirm the accuracy and
SSP properties of the new methods.

The ideas here are applied to explicit Runge--Kutta methods, but they
could also be applied to other classes of methods including implicit
Runge--Kutta methods, general linear methods, and Rosenbrock methods.

        \section*{Acknowledgments}{
                The authors would like to thank the anonymous referees for their helpful 
                and valuable suggestions on the paper.
        }

        \bibliographystyle{acm} 
        \bibliography{bibliography2}

\end{document}